\documentclass[10pt]{amsart}
\usepackage{amsthm,amsmath,amssymb,amscd,graphics,enumerate, stmaryrd,xspace,verbatim, epic, eepic,url,epigraph}

\usepackage[usenames,dvipsnames]{color}
\usepackage[colorlinks=true]{hyperref}
\makeindex

\usepackage[active]{srcltx}
\usepackage[all]{xypic}
\SelectTips{cm}{}

{
   \newtheorem{theorem}[subsubsection]{Theorem}
      \newtheorem*{theorem*}{Theorem}
   \newtheorem{proposition}[subsubsection]{Proposition}
   
   \newtheorem{lemma}[subsubsection]{Lemma}

   \newtheorem{corollary}[subsubsection]{Corollary}
   
   \newtheorem*{conjecture*}{Conjecture}
   
}
{\theoremstyle{definition}
          \newtheorem*{exercise*}{Exercise}
   
   \newtheorem{example}[subsubsection]{Example}
   \newtheorem*{example*}{Example}

   \newtheorem*{definition*}{Definition}
   \newtheorem{rem}[subsubsection]{Remark}
   \newtheorem{remark}[subsubsection]{Remark}

}
\newcommand{\RR}{{\mathbb{R}}}

\newcommand{\CC}{{\mathbb{C}}}
\newcommand{\QQ}{{\mathbb{Q}}}
\newcommand{\NN}{{\mathbb{N}}}

\newcommand{\PP}{{\mathbb{P}}}
\newcommand{\ZZ}{{\mathbb{Z}}}

\newcommand{\GG}{{\mathbb{G}}}
\newcommand{\LL}{{\mathbb{L}}}
\newcommand{\bbA}{{\mathbb{A}}}
\renewcommand{\AA}{{\mathbb{A}}}
\newcommand\Lam{{\Lambda}}
\newcommand{\cO}{{\mathcal O}}

\newcommand{\oB}{{\overline{B}}}
\newcommand{\oO}{{\overline{O}}}

\newcommand{\of}{{\overline{f}}}
\newcommand{\og}{{\overline{g}}}
\newcommand{\oh}{{\overline{h}}}

\newcommand{\ophi}{{\overline{\phi}}}

\def\<{\langle}
\def\>{\rangle}

\newcommand{\Spec}{\operatorname{Spec}}

\newcommand{\Proj}{\operatorname{Proj}}

\newcommand{\Ker}{{\operatorname{Ker}}}

\newcommand{\Tor}{{\operatorname{Tor}}}

\newcommand{\chara}{\operatorname{char}}

\newcommand{\oD}{{\overline{D}}}

\def\:{{\colon}}
\def\.{{,\dots,}}
\def\veps{\varepsilon}

\newcommand{\double}{\genfrac..{0pt}1
{\raise -1pt\hbox{$\scriptstyle\longrightarrow$}}{\raise 3pt\hbox
{$\scriptstyle\longrightarrow$}}}

\renewcommand{\setminus}{\smallsetminus}

\def\tototi{\mathbin{\mathop{\otimes}\limits^{\raise-1pt\hbox
{$\scriptscriptstyle {\rm L}$}}}}

\def\indlim{\mathop{\vrule width0pt height7pt depth
4pt\smash{\lim\limits_{\raise 1pt\hbox to 14.5pt
{\rightarrowfill}}}}}
\def\projlim{\mathop{\vrule width0pt height7pt depth
4pt\smash{\lim\limits_{\raise 1pt\hbox to 14.5pt
{\leftarrowfill}}}}}

\newcommand\displaceamount{3pt}

\newcommand{\doubledown}{\ar@<\displaceamount>[d]\ar@<-\displaceamount>[d]}

\newcommand{\doubleup}{\ar@<\displaceamount>[u]\ar@<-\displaceamount>[u]}

\newcommand{\doubleright}{\ar@<\displaceamount>[r]\ar@<-\displaceamount>[r]}

\def\into{\hookrightarrow}
\def\onto{\twoheadrightarrow}

\def\GGm{{\mathbb G}_m}

\def\toisom{\xrightarrow{{_\sim}}}

\def\hatA{{\widehat A}}

\def\hatB{{\widehat B}}
\def\hatK{{\widehat K}}
\def\hatR{{\widehat R}}
\def\hatf{{\widehat f}}
\def\bfD{{\mathbf{D}}}
\def\wtimes{{\widehat\otimes}}

\def\hatphi{{\widehat\phi}}

\def\Frac{{\rm Frac}}

\begin{document}
\title{Luna's fundamental lemma for diagonalizable groups}

\author[D. Abramovich]{Dan Abramovich}
\address{Department of Mathematics, Box 1917, Brown University,
Providence, RI, 02912, U.S.A}
\email{abrmovic@math.brown.edu}
\author[M. Temkin]{Michael Temkin}
\address{Einstein Institute of Mathematics\\
               The Hebrew University of Jerusalem\\
                Giv'at Ram, Jerusalem, 91904, Israel}
\email{temkin@math.huji.ac.il}
\thanks{This research is supported by  BSF grant 2010255}

\date{\today}

\begin{abstract}
We study relatively affine actions of a diagonalizable group $G$ on locally noetherian schemes. In particular, we generalize  Luna's fundamental lemma when applied to a diagonalizable group: we obtain criteria for a $G$-equivariant morphism $f: X'\to X$ to be {\em strongly equivariant}, namely the base change of the morphism $f\sslash G$ of quotient schemes, and establish descent criteria for $f\sslash G$ to be an open embedding, \'etale, smooth, regular, syntomic, or lci.
\end{abstract}

\maketitle
\setcounter{tocdepth}{1}

\tableofcontents

\section{Introduction}

\subsection{Luna's Fundamental Lemma}
In its original formulation, Luna's Fundamental Lemma  \cite[Lemme fondamental]{Luna} addresses the following classical question: we are given an \'etale equivariant morphism $f:X' \to X$ between two {\em normal complex varieties} with the action of a reductive algebraic group $G$. We assume there are quotients $ Y =  X\sslash G$ and $Y'  =X'\sslash G$ with a resulting quotient morphism $f\sslash G :   Y' \to Y$. Under what conditions is $f$ {\em strongly \'etale, namely $X' = Y' \times_{Y} X$ and $f\sslash G$ is \'etale?} In other words, when is the following diagram cartesian and the bottom row  \'etale as well?

$$\xymatrix{X' \ar[r]^f\ar[d]_{q'} & X\ar[d]^{q}\\ Y' \ar[r]_{f\sslash G} & Y}
$$

\subsection{Generalized context}
In this paper we restrict attention to {\em diagonalizable groups}, relax all the other conditions, and deduce further properties of the quotient:
\begin{enumerate}
\item Our schemes are only required to be locally noetherian, with no assumption on being normal or even reduced.
\item The action of $G$ on $X$ and $X'$ is assumed to be {\em relatively affine} in the sense of Section  \ref{relaffinesec}.
\item We provide necessary and sufficient conditions for an arbitrary equivariant morphism $f$ to be {\em strongly equivariant} in the sense that $f$ is the base change of $f\sslash G$.
\item We establish descent of various properties of $f$, ensuring that $f\sslash G$ is regular, smooth, \'etale, or syntomic (without any finite type assumption), when $f$ has the corresponding property.
\end{enumerate}
Condition (2) guarantees the existence of good quotients $Y$ and $Y'$, see Section~\ref{relaffinesec}. It also guarantees that over any point $y \in Y$ the fiber $X_y$ contains a unique closed orbit - called here a {\em special orbit} (Definition \ref{Def:special-orbit}).

Already in the complex case, most equivariant morphisms are not strongly equivariant: one must require something about how $f$ treats stabilizers  and special orbits. Accordingly, we say that $f:X'\to X$ is {\em fiberwise inert} (\ref{Sec:inert-def}) if it sends special orbits to special orbits, and if for each point $x\in X$ the stabilizer $G_x$ acts trivially on the fiber $X'\times_X x$.

\subsection{The main result}
The main result of this article is the following variant of Luna's Fundamental Lemma

\begin{theorem}[See Theorem \ref{Th:Luna}]\label{Th:maintheorem}
Let $X$ and $X'$ be locally noetherian schemes provided with relatively affine actions of a diagonalizable group $G$, and let $f\:X'\to X$ be a $G$-equivariant morphism. Then,

(i) $f$ is strongly equivariant if and only if $f$ is fiberwise inert and any special orbit in $X'$ contains a point $x'$ such that the action of $G_{x'}$ on $H_1(\LL_{X'/X}\otimes^{\rm L}k(x'))$ is trivial.

(ii) Let $P$ be one of the following properties: (a) regular, (b) smooth, (c) \'etale, (d) an open embedding. Then the following conditions are equivalent: (1) $f$ is inert and satisfies $P$, (2) $f$ is strongly equivariant and satisfies $P$, (3) $f$ strongly satisfies $P$.

(iii) $f$ is strongly syntomic if and only if it is syntomic and strongly equivariant. Moreover, if $f$ is strongly equivariant then the following claims hold: (a) if $f$ is lci then $f\sslash G$ is lci, (b) if $f\sslash G$ is lci and $\Tor$-independent with the morphism $X \to X\sslash G$ then $f$ is lci.
\end{theorem}

Part (iii) refers to Avramov's definition of lci morphisms and the resulting notion of syntomic morphisms described in Section \ref{Sec:lci-syntomic}.

Theorem \ref{Th:maintheorem} is the first ingredient in our forthcoming work  \cite{AT2} on weak factorization of birational maps, generalizing the main theorems of \cite{AKMW} and \cite{W-Factor} to the appropriate generality of qe schemes, and  further proving factorization results in other geometric categories of interest.
While \cite{AT2} only requires actions of $\GG_m$, we find it both convenient and fruitful to work with an arbitrary diagonalizable group. A second ingredient  for  \cite{AT2} - called {\em torification} - is developed in the appropriate generality in \cite{AT1}, which builds on this paper.

\subsection{Comparison with the literature}
Luna's original work applied to normal complex varieties with reductive group actions, and it was extended to algebraically closed field of arbitrary characteristics by Bardsley and Richardson, see \cite{Bardsley-Richardson}. In both works, the only condition one has to impose on an \'etale equivariant morphism $f$ is pointwise inertia preservation.

Recently Alper generalized Luna's Fundamental Lemma to a much more general context of good moduli spaces of Artin stacks, see \cite[Theorem~6.10]{Alper-local-structure}. In this situation, the morphism from a stack to its moduli space is an analogue of the morphism $X\to X\sslash G$, and Alper studied the classical question of describing strongly \'eale morphisms. In this generality, an additional condition has to be imposed on $f$, and the weak saturation condition introduced in \cite[Section~2.2]{Alper-local-structure}) is analogous to our condition on special orbits.

\subsection{Further directions}
It would be interesting, perhaps in future work,  to further extend Luna's Fundamental Lemma  to other tame actions, i.e. actions having  linearly reductive (or even reductive) stabilizers and affine orbits. Moreover, one may hope to extend this to tame groupoids and their quotients, that one may call tame stacks (if the stabilizers are of dimension zero then those are the tame stacks of \cite{AOV1}). We note, however, that a simple descent argument shows that all results of this paper hold for groups of multiplicative type, see Section \ref{nonsplitsec}.

\subsection{Auxiliary results} On the way to proving Theorem  \ref{Th:maintheorem}  we study related notions of regularity, formal smoothness, and group actions. Some results which may be of independent interest and will be used in \cite{AT1} are recorded here. The reader interested only in Theorem \ref{Th:maintheorem} may wish to skip directly to Section \ref{section:reg}.

\subsubsection{Splitting of formally smooth morphisms} We review in Section \ref{section:reg} the notions of regularity and formal smoothness. In particular, Section \ref{mixedcase} introduces the notion of a Cohen ring $C(k) \to k$ of a field $k$; further, by
\cite[tag/032A]{stacks} any formally smooth $g:k \to D$ with $D$ a complete noetherian $k$-algebra lifts to a formally smooth morphism  we denote by $C(g): C(k) \to C(D)$, where $C(D)$ is a complete noetherian local ring so that $D = C(D) \otimes_{C(k)} k$. This applies in  characteristic 0 by taking $C(k)=k$ and $C(D)=D$. The following result describes an arbitrary formally smooth homomorphism in terms of such $C(g)$.


\begin{theorem}[See Theorem \ref{Th:splitfsmoothmixed}]
Let $f\:A\to B$ be a local homomorphism of complete noetherian local rings with closed fiber $\of\:k=A/m_A\to\oB=B/m_AB$.

(i) Assume $\of$ is formally smooth. Then there exist homomorphisms $i\:C(k)\to A$ and $j\:C(\oB)\to B$ making the following diagram commutative.

\begin{equation*}\xymatrix{ A\ar[rr]^f\ar[dd] && B\ar[dd] \\
 & C(k) \ar@{.>}[ul]^i \ar[rr]^(0.35){C(\of)}|!{[ur];[dr]}\hole\ar[dl] && C(\oB)\ar[dl]\ar@{.>}[ul]^j\\
 k\ar[rr]^\of && \oB}
 \end{equation*}

(ii) Assume $f$ is formally smooth. Then for any choice of $i$ and $j$ the homomorphism $A\wtimes_{C(k)}C(\oB)\to B$ is an isomorphism. In particular, $f$ is (non-canonically) isomorphic to the formal base change of a Cohen lift $C(\of)$ of its closed fiber $\of$.
\end{theorem}

\subsubsection{$L$-local rings} After reviewing groups, actions and diagonalizable groups in Section \ref{Sec:groups} we study diagonalizable group actions on affine schemes in Section \ref{Sec:L-graded}. Consider a finitely generated abelian group $L$ and its Cartier dual $G = \mathbf{D}_L$. An action of $G$ on an affine scheme $X = \Spec A$ is an $L$-grading on $A$. An important class of $L$-graded rings that we study is the class of {\em $L$-local rings}, namely rings $A$ possessing a single maximal $L$-homogeneous ideal $m_A$, see Section \ref{Sec:local}. The residue ring $A/m_A$ has no nontrivial homogeneous ideals, making it a ``graded field".  Several standard results on local rings generalize to this setting, including the following:

\begin{proposition}[Graded Nakayama's Lemma, see Proposition \ref{nakayamaprop} and Corollary \ref{nakayamacor}]
Let $(A,m)$ be an $L$-local ring with residue graded field $A/m_A$ and let $M$ be a finitely generated $L$-graded $A$-module. Then

(i) $mM=M$ if and only if $M=0$.

(ii) A homogeneous homomorphism of $L$-graded $A$-modules $\phi\:N\to M$ is surjective if and only if $\phi\otimes_A (A/m_A)$ is surjective.

(iii) Homogeneous elements $m_1\.m_l$ generate $M$ if and only if their images generate $M/mM$.

(iv) The minimal cardinality of a set of homogeneous generators of $M$ equals to the rank of the free $(A/m_A)$-module $M/mM$.
\end{proposition}

\begin{proposition}[Characterization of equivariant  Cartier divisors, see Proposition \ref{Cartierprop}]
Assume that $(A,m)$ is an $L$-local integral domain and $D\subset\Spec(A)$ is an equivariant finitely presented closed subscheme. Let $x$ be an arbitrary point of $V(m)$ and let $X_x=\Spec(\cO_{X,x})$ be the localization at $x$. Then the following conditions are equivalent:

(i) $D=V(f)$ for a homogeneous element $f\in A$,

(ii) $D$ is a Cartier divisor in $X$,

(iii) $D_x=D\times_XX_x$ is a Cartier divisor in $X_x$.
\end{proposition}

To characterize the situation discussed above, we say that a relatively affine action of $G$ on a scheme $X$ is {\em local} if $X$ is quasi-compact and contains a single closed orbit, see Section \ref{Sec:local-action}.

\begin{proposition}[Characterization of local actions, see Lemma \ref{Lem:local-action}]
Assume we have a relatively affine action of  $G=\bfD_L$  on a scheme $X$. Then the following conditions are equivalent:

(i) The action is local.

(ii) $X$ is affine, say $X=\Spec(A)$, and the $L$-graded ring $A$ is $L$-local.

(iii) The quotient $Y=X\sslash G$ is local.
\end{proposition}

\subsubsection{Completions} By an {\em $L$-complete local ring}\index{Lcomplete local ring@$L$-complete local ring} we mean a complete local ring $(A,m)$ provided with a {\em formal $L$-grading} $A=\prod_{n\in L}A_n$ such that $A_n\subset m$ for each $n\neq 0$. Here is a key result about completions which is used in the paper to prove the formal version of Luna's Fundamental Lemma, see Theorem~\ref{mainformalth}.
\begin{proposition}[Completions of $L$-local rings, see Proposition \ref{Prop:complete}]
Assume that $L$ is a finitely generated abelian group and $(A,m)$ is a noetherian strictly $L$-local ring. Set $m_0=m\cap A_0$ and for each $n\in L$ let $\hatA_n$ denote the $m_0$-adic completion of the $A_0$-module $A_n$. Then, the $m$-adic completion of $A$ is isomorphic to $\prod_{n\in L}\hatA_n$; in particular, it is an $L$-complete local ring.
\end{proposition}

\subsubsection{Quotients and strong equivariance} Using the affine theory as a starting point, a global theory of relatively affine actions is developed in Section~\ref{Lunasection}, leading to the main theorem. We highlight here the following basic results.

\begin{theorem}[Properties preserved by quotients, see Theorem \ref{quottheorem}]
Assume that a diagonalizable group $G=\bfD_L$ acts trivially on a scheme $S$ and an $S$-scheme $X$ is provided with a relatively affine action of $G$, then:

(i) Assume that $X$ satisfies one of the following properties: (a) reduced, (b) integral, (c) normal with finitely many connected components, (d) locally of finite type over $S$, (e) of finite type over $S$, (f) quasi-compact over $S$, (g) locally noetherian, (h) noetherian. Then $X\sslash G$ satisfies the same property.

(ii) If $X$ is locally noetherian then the quotient morphism $X\to X\sslash G$ is of finite type.
\end{theorem}

\begin{proposition}[Preservation of strong equivariance, see Lemma \ref{stronglem}]
Let $G=\bfD_L$ be a diagonalizable group.

(i) The composition of strongly $G$-equivariant morphisms is strongly $G$-equivariant.

(ii) If $Y\to X$ is a strongly $G$-equivariant morphism and $g\:Z\to Y$ is a $G$-equivariant morphism such that the composition is strongly $G$-equivariant, then $g$ is strongly $G$-equivariant.

(iii) If $Y\to X$ is strongly $G$-equivariant and $Z\to X$ is $G$-equivariant then the base change $Y\times_XZ\to Z$ is strongly $G$-equivariant.

(iv) If $f\:Y\to X$ is strongly equivariant then the diagonal $\Delta_f\:Y\to Y\times_XY$ is strongly equivariant and $\Delta_f\sslash G$ is the diagonal of $f\sslash G$.
\end{proposition}


\section{Regularity and formal smoothness}\label{section:reg}
We recall in this section basic facts about formal smoothness in the local case, which is the only case we will use later. The cited results are due to Grothendieck, but we will cite \cite{stacks} in addition to \cite{ega}. In addition, we will prove in Theorems~\ref{Th:splitfsmooth} and \ref{Th:splitfsmoothmixed} a splitting result for local formally smooth homomorphisms, which seems to be new, although it is close in spirit to what was known.

\subsection{Definitions}

\subsubsection{Regular morphisms} Regular morphisms\index{regular morphism} are a generalization of smooth morphisms in situations of morphisms which are not necessarily of finite type. Following \cite[$\rm IV_2$, 6.8.1]{ega}  a morphism of schemes $f\:Y\to X$ is said to be  {\em regular} if
\begin{itemize}
\item  the morphism $f$ is flat and
\item all geometric fibers of $f\: X \to Y$ are regular.
\end{itemize}

\subsubsection{Formal smoothness of local homomorphisms}
Let $f\:(A,m_A)\to(B,m_B)$ be a local homomorphism of local rings. By saying that $f$ is {\em formally smooth} we mean that it is formally smooth\index{formally smooth morphism} with respect to the $m_A$-adic and $m_B$-adic topologies. This means that for any ring $C$ with a square zero ideal $I$ and compatible homomorphisms $A\to C$ and $B\to C/I$ that vanish on large powers of $m_A$ and $m_B$, respectively,
$$
\xymatrix{
A\ar[r]\ar[d]_f & C\ar[d]\\
B\ar[r]\ar@{.>}[ur] &C/I
}
$$
there exists a lifting $B\to C$ making the diagram commutative. Using a slightly non-standard terminology, we say that a morphism of schemes $f\:Y\to X$ is {\em formally smooth at} $y\in Y$ if the local homomorphism $\cO_{X,f(y)}\to\cO_{Y,y}$ is formally smooth.

\begin{remark}
Unlike regularity, formal smoothness is not a local property. Consider an example of a DVR which is not excellent in characteristic $p$: a DVR $R$ such that $\hatK=\Frac(\hatR)$ is not separable over $K=\Frac(R)$. For example, set $\hatR=k[[x]]$, take an element $y\in k[[x]]$ which is transcendental over $k(x)$, and set $R=\hatR\cap k(x,y^p)$. Then $\hatR$ is, indeed, the completion of $R$, and the homomorphism $R\to\hatR$ is formally smooth but its generic fiber $K\to\hatK$ is not.
\end{remark}

\subsubsection{Formally factorizable homomorphisms}\label{factorizablesec}
Following \cite{Franco-Rodicio}, we say that a homomorphism of noetherian local rings $\phi\:A\to B$ is {\em formally factorizable} if its completion can be factored into a composition of a formally smooth homomorphism of noetherian complete local rings $\hatA\to D$ and a surjective homomorphism $D\to\hatB$. Any such factorization $\hatA\to D\to\hatB$ will be called a {\em formally smooth factorization} of $\hatphi$. It is proved in \cite[Theorem~4]{Franco-Rodicio} that $\phi$ is formally factorizable if and only if the extension of the residue fields $l/k$ has a finite-dimensional imperfection module $\Upsilon_{l/k}=H_1(\LL_{l/k})$ (see \cite[$\rm 0_{IV}$, 20.6.1 and 21.4.8]{ega}).\index{formally factorizable homomorphism}

\begin{remark}
Note that $\dim_l\Upsilon_{l/k}=\infty$ can happen only when $k$ has infinite $p$-rank (i.e. $[k:k^p]=\infty$) and the extension $l/k$ is not finitely generated. In particular, formally non-factorizable homomorphisms are pretty exotic and almost never show up in applications.
\end{remark}

\subsection{Properties of formal smoothness}

\subsubsection{Criteria}
Here are the main criteria for formal smoothness of local homomorphisms of noetherian rings.

\begin{theorem}\label{Th:fsmoothcrit}
For a local homomorphism $f\:(A,m)\to(B,n)$ of noetherian local rings the following conditions are equivalent:

(i) $f$ is formally smooth,

(ii) the completion $\hatf\:\hatA\to\hatB$ is formally smooth,

(ii)' the completion $\hatf\:\hatA\to\hatB$ is a regular homomorphism,

(iii) $f$ is flat and its closed fiber $\of\:k=A/m\to\oB=B/mB$ is formally smooth,

(iii)' $f$ is flat and its closed fiber $\of\:k=A/m\to\oB=B/mB$ is a regular homomorphism (i.e. $\oB$ is geometrically regular over $k$).
\end{theorem}
\begin{proof}
The equivalence of (i) and (ii) is almost obvious, see \cite[{\tt tag/07ED}]{stacks}. See \cite[{\tt tag/07NQ}]{stacks} for the equivalence of (ii), (iii), and (iii)', and see \cite[{\tt tag/07PM}]{stacks} for the equivalence of (ii) and (ii)'.
\end{proof}

\begin{corollary}\label{Cor:regfsmooth}
A morphism of noetherian schemes $f\:Y\to X$ is regular if and only if it is formally smooth at all points of $Y$.
\end{corollary}
\begin{proof}
First assume $f$ is formally smooth at all points of $Y$, so Statement (i) above is satisfied locally at every point. By the theorem,  Statement  (iii)' is satisfied locally at every point, so $f$ is flat and the geometric  fibers are locally regular, hence regular.  Being flat with regular geometric fibers, the morphism $f$ is regular.

The other direction is similar: assume $f$ is regular, so it satisfies Statment (iii)' of the theorem at every point. By the theorem,  Statement (i) holds at every point, namely $f$ is formally smooth at every point.
\end{proof}

As another corollary, we obtain that formal smoothness behaves nicely for qe (or quasi-excellent) schemes (e.g., see \cite[Definition~I.2.10]{Illusie-Temkin}). In particular, it becomes a local property (this is due to Andr\'e, see \cite{Andre}).

\begin{corollary}
Assume that $\phi\:A\to B$ is a local homomorphism of noetherian rings and $A$ is a qe ring. Then $\phi$ is formally smooth if and only if it is regular. In particular, if $f\:Y\to X$ is a morphism of noetherian schemes with $X$ a qe scheme, and if $f$ is formally smooth at all closed points of $Y$, then $f$ is regular.
\end{corollary}
\begin{proof}
Assume $\phi$ is formally smooth. By Theorem~\ref{Th:fsmoothcrit}, the completion $\hatphi\:\hatA\to\hatB$ is regular. Since $A$ is quasi-excellent, the composition $A\to\hatA\to\hatB$ is regular. Since $B$ is noetherian, the completion homomorphism $B\to\hatB$ is flat and surjective on spectra.  By \cite[{\tt tag/07NT}]{stacks}  the homomorphism $\phi$ is regular.

The other direction follows from Theorem \ref{Th:fsmoothcrit} or Corollary \ref{Cor:regfsmooth}.
\end{proof}

\subsubsection{Adic lifting property}
It is shown in \cite[{\tt tag/07NJ}]{stacks} that a formally smooth $f$ satisfies a strong lifting property with respect to continuous homomorphisms to adic rings. We will need the following particular case:

\begin{lemma}\label{Lem:liftadic}
Let $f\:A\to B$ be a formally smooth local homomorphism, let $C$ be a complete noetherian local ring, and let $I\subset C$ be any ideal. Then any pair of compatible homomorphisms of topological rings $A\to C$ and $B\to C/I$ admits a lifting $B\to C$.
\end{lemma}
\begin{proof}
Since $C$ is noetherian $I$ is a closed ideal, see \cite[Theorem 8.14]{Matsumura-ringtheory}. Also $I$ is contained in the maximal ideal of $C$, which is an ideal of definition. Therefore the claim follows from \cite[{\tt tag/07NJ}]{stacks}.
\end{proof}

\begin{corollary}\label{Cor:liftadic}
Let $A$ be a local ring with residue field $k$, let $B$ and $C$ be complete noetherian local $A$-algebras, and assume that $B$ is formally smooth over $A$. Then an $A$-homomorphism $g\:C\to B$ is an isomorphism if and only if its closed fiber $\og=g\otimes_A k$ is an isomorphism.
\end{corollary}
\begin{proof}
Only the inverse implication needs a proof. We first claim that when the closed fiber $\og$ is surjective then $g$ is surjective.  Since $B$ and $C$ are complete with respect to their maximal ideals, they are also complete with respect to $m_AB$ and $m_AC$, respectively. As $\og$ is surjective, the $(n-1)$st infinitesimal fibers $C/m_A^nC\to B/m_A^nB$ are surjective by Nakayama's lemma for nilpotent ideals, and we obtain that $g$ is surjective too.

 Setting $I=\Ker(g)$ and using Lemma~\ref{Lem:liftadic}, we obtain that there exists a homomorphism $h\:B\to C$ making the following diagram commutative
$$
\xymatrix{
A\ar[r]\ar[d] & C\ar@{->>}[d]^g\\
B\ar@{=}[r]\ar@{.>}[ur]^h &C/I
}
$$
In other words, we have found a section $h$ of $g$. Note that the closed fiber of $h$ is surjective (in fact, $\oh=\og^{-1}$). The first statement proven above with $g$ replaced by $h$ gives that $h$ is surjective.
\end{proof}

\subsubsection{Splitting of formally smooth homomorphisms}
Now we are ready to prove the main result of section~\ref{section:reg} in the equal characteristic case: any formally smooth homomorphism between complete noetherian local rings is a base change of its closed fiber. For convenience of the exposition we will deal with the mixed characteristic case separately.

Recall that if $A$ is a complete noetherian local ring with residue field $k$, and if $A$ contains a field, then $A$ contains $k$ as a coefficient field $i\:k\into A$  by Cohen's Structure Theorem, \cite[{\tt tag/032A}]{stacks}.
\begin{theorem}\label{Th:splitfsmooth}
Let $f\:A\to B$ be a local homomorphism of complete noetherian local rings with closed fiber $\of\:k=A/m_A\to\oB=B/m_AB$. Assume that $A$ contains a field.

(i) Assume $\of$ is formally smooth. Then there is a coefficient field $i\:k\to A$ and a section $j\:\oB\into B$  of the surjection $B\onto\oB$ which extends the composition $f\circ i\:k\into A\to B$.

(ii) Assume $f$ is formally smooth. Then for any choice of such $i$ and $j$, the homomorphism $A\wtimes_k\oB\to B$ is an isomorphism.
In particular, $f$ is (non-canonically) isomorphic to the formal base change of its closed fiber $\of$ with respect to $i$.
\end{theorem}
\begin{proof}
To establish existence of $j$ we apply Lemma~\ref{Lem:liftadic} to the following diagram:
$$
\xymatrix{
k\ar[d]_{\of}\ar[r]^{f\circ i} & B\ar@{->>}[d]\\
\oB\ar@{=}[r]\ar@{.>}[ur]^j &\oB
}
$$
Once $j$ is fixed, we obtain a homomorphism $g\:A\wtimes_k\oB\to B$ of complete noetherian local $A$-algebras whose closed fiber $\og\:k\otimes_k\oB\toisom\oB$ is an isomorphism. Hence $g$ is an isomorphism by Corollary~\ref{Cor:liftadic}.
\end{proof}

\subsubsection{The mixed characteristic case}\label{mixedcase}
Recall that given a field $k$ with $\chara(k)=p>0$, a Cohen ring $C(k)$ is a complete DVR with residue field $k$ and maximal ideal $(p)$. Since $C(k)$ is formally smooth over $\ZZ_p$ (see Theorem~\ref{Th:fsmoothcrit}), given a complete local ring $A$ with $A/m_A=k$, the homomorphism $C(k)\to k$ lifts to $f\:C(k)\to A$. In fact, this argument is used in the proof of Cohen's Structure Theorem, \cite[{\tt tag/032A}]{stacks}: in the mixed characteristic case, $C(k)\into A$ is a ring of coefficients of $A$, and in the equal characteristic case $Im(f)=k$ is a field of coefficients of $A$.

Assume  $\chara(k)$ is of characteristic $p>0$. By \cite[{\tt tag/07NR}]{stacks}, any formally smooth homomorphism $g\:k\to D$ with $D$ a complete noetherian $k$-algebra admits a formally smooth lifting $f\:C(k)\to E$, with $E$ a suitable complete noetherian ring,  in the sense that $g=f\otimes_{C(k)} k$. Moreover, if $f'\:C(k)\to E'$ is another such lifting then by formal smoothness of $f$ the homomorphism $E\to D$ lifts to a homomorphism $E\to E'$, which is necessarily an isomorphism. For this reason, we will use the notation $C(D)=E$ and $C(g)=f$.

In order to unify the notation, if $R$ is a ring containing $\QQ$, we set $C(R)=R$, and for any homomorphism $f\:R\to S$ we set $C(f)=f$. Here is the analogue of Theorem~\ref{Th:splitfsmooth}.

\begin{theorem}\label{Th:splitfsmoothmixed}
Let $f\:A\to B$ be a local homomorphism of complete noetherian local rings with closed fiber $\of\:k=A/m_A\to\oB=B/m_AB$.

(i) Assume $\of$ is formally smooth. Then there exist homomorphisms $i\:C(k)\to A$ and $j\:C(\oB)\to B$ making the following diagram commutative.

\begin{equation}\label{Eq:splitfsmoothmixed}\xymatrix{ A\ar[rr]^f\ar[dd] && B\ar[dd] \\
 & C(k) \ar@{.>}[ul]^i \ar[rr]^(0.35){C(\of)}|!{[ur];[dr]}\hole\ar[dl] && C(\oB)\ar[dl]\ar@{.>}[ul]^j\\
 k\ar[rr]^\of && \oB}
 \end{equation}

(ii) Assume $f$ is formally smooth. Then for any choice of $i$ and $j$ the homomorphism $A\wtimes_{C(k)}C(\oB)\to B$ is an isomorphism. In particular, $f$ is (non-canonically) isomorphic to the formal base change of a Cohen lift $C(\of)$ of its closed fiber $\of$.
\end{theorem}
\begin{proof}
As we saw in the beginning of Section \ref{mixedcase}, the homomorphism $C(k)\to k$ lifts to $i\:C(k)\to A$. In particular, $B$ becomes a $C(k)$-algebra. Since $C(\of)$ is formally smooth by Theorem~\ref{Th:fsmoothcrit}, the $C(k)$-homomorphism $C(\oB)\to\oB$ lifts to a $C(k)$-homomorphism $j\:C(\oB)\to B$.

Given $i$ and $j$, we obtain a homomorphism $g\:A\wtimes_{C(k)}C(\oB)\to B$ of complete noetherian local $A$-algebras whose closed fiber $\og\:k\otimes_k\oB\toisom\oB$ is an isomorphism. Hence $g$ is an isomorphism by Corollary~\ref{Cor:liftadic}.
\end{proof}

\begin{remark}\label{formalrem}
In a sense, Theorem \ref{Th:splitfsmooth} reduces the classification of formally smooth homomorphisms $f\:A\to B$ of complete local rings to the case when the source is a field. By Theorem~\ref{Th:fsmoothcrit}, $g\:k\to\Lam$ is formally smooth if and only if $\Lam$ is geometrically regular over $k$. Perhaps, this is the ``best'' characterization of formally smooth $k$-algebras one can give in general. On the other hand, if we further assume that $K=\Lam/m_\Lam$ is separable over $k$ (e.g., if $k$ is perfect) then a better characterization is possible: $g$ is formally smooth if and only if $\Lam$ is of the form $K\llbracket t_1,\dots,t_n\rrbracket$. Indeed, since  $k\to K$ is formally smooth we can extend $g$ to a field of coefficients $K\into\Lam$. Then we choose $t_1\.t_n$ to be any family of regular parameters. Applying Theorem~\ref{Th:splitfsmoothmixed} we obtain as a consequence that if $f$ is formally smooth and the extension of the residue fields $K/k$ is separable then there is an isomorphism of $A$-algebras $A\wtimes_{C(k)}C(K)\llbracket t_1\.t_n\rrbracket\toisom B$.
\end{remark}

\section{Generalities on group scheme actions}\label{Sec:groups}
In this section we fix some basic terminology, including group schemes, orbits, stabilizers, etc.

\subsection{General groups}

\subsubsection{Group schemes and actions}
A $\ZZ$-flat groupscheme (resp. flat $S$-group scheme) $G$ will be simply referred to as a {\em group} (resp. {\em $S$-group}). An {\em action} of a group $G$ on a scheme $X$ is given by an {\em action morphism} $\mu\:G\times X\to X$ satisfying the usual compatibilities: associativity and the triviality of the unit action. Similarly an action of an $S$-group $G$ on an $S$-scheme  $X$ is given by a morphism $\mu\:G\times_S X\to X$ satisfying the analogous requirements.  If $X$ is an $S$-scheme and $G$ is a group then an action of $G$ on $X$ is called an {\em $S$-action} if $\mu$ is an $S$-morphism. Giving such an action is equivalent to providing $X$ with an action of the $S$-group $G_S=G\times S$.

\begin{remark}
The projection and the action morphisms give rise to an fppf groupoid $G\times X\rightrightarrows X$, see \cite[{\tt tag/0234}]{stacks}, which often shows up in constructions related to the action.
\end{remark}

\subsubsection{Stabilizers}
	The {\em stabilizer} or {\em inertia group}\index{inertia!group} of an action of $G$ on $X$ is the $X$-group scheme $$I_X\ \ =\ \ {\rm Eq}(G_X\rightrightarrows X)\ \ =\ \ G_X\ \mathop\times\limits_{X\times X}\ X,$$ where the component maps of $G_X \to X\times X$ are the action and projection maps and $X \to X\times X $ on the right is the diagonal.  This is a subgroup of the $X$-group $G_X$, which is often not flat. For any point $x\in X$, we define its {\em stabilizer} as the fiber $G_x=I_X\times_X\Spec(k(x))$.

\subsubsection{Orbits}
Working with varieties one usually considers only ``classical" orbits of group actions, which can be characterized as orbits of closed points or locally closed orbits. When $G$ acts on a more general scheme $X$ it is more natural to take into account orbits of all points. A set-theoretic orbit of $x\in X$ is the image of the map $G\times\Spec(k(x))\to X$. This definition ignores the non-reduced structure which becomes essential when $G$ is non-reduced. For example, free and non-free actions of $\mu_p$ are distinguished by the nilpotent structure of the orbits.

In order to define scheme-theoretic orbits one should use the scheme-theoretic image, see \cite[II, Exercise 3.11 (d)]{Hartshorne}, \cite[{\tt tag/01R6}]{stacks}. Let $\oO_x$ be the scheme-theoretic image of $G\times\Spec(k(x))\to X$, and let $O_x$ be obtained from $\oO_x$ by removing all proper closed subsets of the form $\oO_y$. We provide $\oO_x$ with the structure sheaf $\cO_{\oO_x}|_{O_x}$ and call it the {\em $G$-orbit} of $x$. We do not know general criteria for $O_x$ to be a scheme, but the orbits we will use below are in fact schemes. Note that in this case, $O_x$ is a limit of open subschemes of $\oO_x$.

\subsection{Diagonalizable groups}
Starting from this point, we consider only diagonalizable groups $G$. Probably, many results can be extended to the case of arbitrary linearly reductive or even reductive groups, but we do not pursue that direction.

\subsubsection{The definition}
By a {\em diagonalizable group} we mean a finite type diagonalizable group\index{diagonalizable group} $G$ over $\ZZ$, see \cite[VIII.1.1]{SGA3-2}. In other words, $G=\mathbf{D}_L=\Spec(\ZZ[L])$ for a finitely generated abelian group $L$. Note that for any subgroup $L'\subseteq L$ with the quotient $L''=L/L'$, we have a natural embedding $\mathbf{D}_{L''}\into \mathbf{D}_L$ and $\mathbf{D}_L/\mathbf{D}_{L''}\toisom \mathbf{D}_{L'}$. Moreover, this construction exhausts all subgroups and quotient groups of $\mathbf{D}_L$.

\subsubsection{Action of a diagonalizable group}
Any element $d\in L$ induces a character $\chi_d\:G\to\GGm=\mathbf{D}_\ZZ$, and this construction identifies $L$ with the group of all characters of $G$. An action of a diagonalizable group $G=\mathbf{D}_L$ on a scheme $X$ can also be described in the dual language by giving a comultiplication homomorphism of $\cO_X$-algebras $\mu^\#\:\cO_X\to\cO_X[L]$.

\subsubsection{The affine case}
If $X=\Spec(A)$ is affine then the action is described by the homomorphism $\mu^\#\:A\to A[L]$, and it is easy to see that such a homomorphism $\mu^\#$ is, indeed, a comultiplication if and only if it corresponds to an $L$-grading $A=\oplus_{n\in L}A_n$ on $A$, see \cite[I.4.7.3]{SGA3-1}; if $a=\sum a_n$ with $a_n\in A_n$ then $\mu^\#(a)=\sum a_n n$. In the sequel, we will identify $G$-actions on $X=\Spec(A)$ with the corresponding $L$-gradings of $A$.

\section{$L$-graded rings}\label{Sec:L-graded}
In this section we study how a diagonalizable group $G=\bfD_L$ acts on affine schemes $X=\Spec(A)$. On the algebraic side, this corresponds to studying $L$-graded rings $A=\oplus_{n\in L}A_n$.

\subsection{Coinvariants and the scheme of fixed points}

\subsubsection{Coinvariants}
Given an $L$-graded ring $A$ consider the ideal $I$ generated by all modules $A_n$ with $n\neq 0$. It is a graded ideal, and one has $I_n=A_n$ for $n\neq 0$ and $I_0=\sum_{n\neq 0}A_nA_{-n}$. Note that $A/I=A_0/I_0$ is the maximal graded quotient of $A$ with trivial $L$-grading. We call it the {\em ring of coinvariants} of $A$ and denote it $A_G$ or $A_L$.

\subsubsection{The scheme of fixed points}\label{Sec:fixed}
If $X=\Spec(A)$ then $X^G:=\Spec(A_G)$ is the maximal closed subscheme of $X$ on which the action is trivial. Obviously, $X\mapsto X^G$ is a functor on the category of affine $G$ schemes. We call it the {\em fixed points} functor, see also \S\ref{fixedpointssec} below.

\subsection{Invariants and the quotient}

\subsubsection{The definition}
If $A$ is an $L$-graded ring then $G=\bfD_L$ acts on $A$, and $A_d$ is the set of all elements on which $G$ acts through $\chi_d$. In particular, the ring of invariants $A^{G}$ coincides with $A_0$. The {\em quotient}\index{quotient} of $X=\Spec(A)$ by the action is the scheme $X\sslash G:=\Spec(A_0)$. Obviously, $X\mapsto X\sslash G$ is a functor on the category of affine $G$ schemes. We call it the {\em quotient functor}.

\subsubsection{Some properties preserved by the quotient functor}
It is classical, that the properties of being reduced, integral, and normal are preserved by quotients.

\begin{lemma}\label{invarquotlem}
Assume that an $L$-graded ring $A=\oplus_{n\in L}A_n$ satisfies one of the following properties: reduced, an integral domain, a normal domain. Then $A_0$ satisfies the same property.
\end{lemma}
\begin{proof}
Only the last case needs a justification. Assume that $A$ is a normal domain and so $A_0$ is a domain. If $a,b\in A_0$ are such that $b\neq 0$ and $c=\frac{a}{b}$ is integral over $A_0$ then $c\in A$. Since $bc=a$ in the domain $A$, and $a,b\in A_0$, we obtain that $c$ is also homogeneous of degree zero.
\end{proof}

Another property preserved by quotients is being of finite type.

\begin{lemma}\label{fgquotlem}
Assume that an $L$-graded ring $A=\oplus_{n\in L}A_n$ is finitely generated over a subring $C\subseteq A_0$. Then $A_0$ is finitely generated over $C$.
\end{lemma}
\begin{proof}
We can choose homogeneous $C$-generators of $A$, and this gives a presentation of $A$ as a quotient of an $L$-graded polynomial algebra $B=C[t_1\.t_l]$ by a homogeneous ideal $I$. Then $A_n=B_n/I_n$ for any $n\in L$, and so $A_0$ is a quotient of $B_0$. Note that $B_0=C[M]$ for the monoid $M= \phi^{-1}(0)$ with  $\phi\:\NN^l\to L$ given by $\phi(a_1\.a_l)=\sum_{i=1}^l a_i\deg(t_i)$. Since the monoid $M$ is finitely generated, $B_0$ is finitely generated over $C$, and the lemma follows.
\end{proof}

\subsubsection{Universality of quotients and fixed point schemes}
As opposed to the case of general reductive groups, the quotient is universal in the sense of \cite[Ch. 0, \S1]{GIT}, regardless of the characteristic:

\begin{lemma}\label{Lem:univers}
Assume that $G=\bfD_L$ acts on $X=\Spec(A)$, write $Y=X\sslash G$,  let $Y'\to Y$ be an affine morphism, and let $X'=Y'\times_YX$. Then $X'\sslash G=Y'$ and $(X')^G=X^G\times_XX'$.
\end{lemma}
\begin{proof}
Let $Y'=\Spec(B_0)$. Since $Y=\Spec(A_0)$, we obtain that $X'=\Spec(B)$, where $B=B_0\otimes_{A_0}A$ with the grading $B=\oplus_{n\in L}B_0\otimes_{A_0}A_n$. So, $B_0$ is the degree zero component of $B$, and hence $X'\sslash G=Y'$. In addition, $A_nB=B_nB$ for any $n\in L$, hence $B_G=B\otimes_AA_G$ giving $$(X')^G=\Spec B_G = \Spec  B\otimes_AA_G = X^G\times_XX'.$$
\end{proof}

\begin{remark}
A much more general result is proved by Alper in \cite[Prop. 4.7(i)]{Alper}.
\end{remark}

\subsubsection{Fibers}
The quotient morphism $X\to Y=X\sslash G$ is $G$-equiva\-riant with respect to the trivial action on $Y$, hence it contracts the orbits of the action on $X$. It is well known that two orbits of closed points are mapped to the same point if and only if their closures intersect, so on the set-theoretical level one can view $Y$ as the ``separated" quotient of $X$. For example, if $X$ is defined over a field then the same fact is proved for any reductive group action in \cite[Cor. 1.2]{GIT}. In the case of a diagonalizable group, one can obtain a more precise description as follows.

\begin{lemma}\label{Lem:fib}
Assume that $G=\bfD_L$ acts on $X=\Spec(A)$ and let $y$ be a point of $Y=X\sslash G$. Then the fiber $X_y$ contains a single orbit $O$ which is closed in $X_y$, and this orbit belongs to the closure of any other orbit contained in $X_y$.
\end{lemma}
\begin{proof}
By Lemma~\ref{Lem:univers} we can replace $Y$ with $y=\Spec(k(y))$ and $X$ with $X_y=X\times_Yy$. So, we can assume that $Y=\Spec(k)$ for a field $k$ and $X=\Spec(A)$ for an $L$-graded $k$-algebra $A=\oplus_{n\in L}A_n$ with $A_0=k$. The set $L'$ of elements $n\in L$, such that $A_n$ contains a unit of $A$, is a subgroup of $L$. Furthermore, if $n\in L'$ then any non-zero element of $A_n$ is a unit because $A_0$ is a field. So, $I=\oplus_{n\in L\setminus L'}A_n$ is the maximal {\em homogeneous} ideal of $A$. The closure of any orbit is given by a homogeneous ideal, hence contains $O=\Spec(A/I)$. It remains to observe that $O\toisom\Spec(k[L'])$ is a single orbit with stabilizer $\mathbf{D}_{L/L'}$.
\end{proof}

\begin{example}
If $G=\GGm$ then there are two types of fibers: either $X_y$ contains a single $G$-invariant closed point, or $X_y$ is a single orbit with stabilizer $\mu_n$.
\end{example}

Recall that a morphism $X\to Y$ is {\em submersive} if $U\subset Y$ is open if and only if its preimage is open \cite[{\tt tag/0406}]{stacks}. Lemmas~\ref{Lem:univers} and \ref{Lem:fib}, and \cite[\S0.2, Remarks 5 and 6]{GIT} imply the following result.

\begin{corollary}\label{univcor}
If $G=\bfD_L$ acts on $X=\Spec(A)$ then $X\sslash G$ is a universal categorical quotient and the quotient morphism $X\to X\sslash G$ is submersive.
\end{corollary}

\begin{example}\label{borelexam}
For completeness, we note that the above theory completely breaks down for non-reductive groups. The classical example is obtained when $X=\Spec(A)$ is ${\rm GL}_2(k)$ and $G$ is a Borel subgroup acting on $X$ on the left. Then $A^G=k$, but the categorical quotient is $\PP^1_k$. In particular, $\Spec(A^G)$ is just the affine hull of the categorical quotient.
\end{example}

\subsection{Noetherian $L$-graded rings}
A theorem of S. Goto and K. Yamagishi states that any noetherian $L$-graded ring is finitely generated over the subring of invariants, see \cite{Goto-Yamagishi}. In the case of a finite group action (not necessarily commutative) an analogous claim was recently proved by Gabber, see \cite[Exp. IV, Prop. 2.2.3]{Illusie-Temkin}. It seems that the work \cite{Goto-Yamagishi} is not widely known in algebraic geometry; at least, we have reproved the theorem (in a more complicated way!) before finding the reference. For the sake of completeness, we outline the proof of \cite[Theorem 1.1]{Goto-Yamagishi} below.

\begin{proposition}[Goto-Yamagishi]\label{noetherprop}
Assume that $A=\oplus_{n\in L}A_n$ is a noetherian $L$-graded ring. Then $A_0$ is noetherian, each $A_n$ is a finitely generated $A_0$-module, and $A$ is a finitely generated $A_0$-algebra.
\end{proposition}
\begin{proof}
If $n\in L$ then for any $A_0$-submodule $M\subseteq A_n$ the ideal $MA$ of $A$ satisfies $MA\cap A_n=M$. It follows that each $A_n$ is a noetherian $A_0$-module, so $A_0$ is a noetherian ring and each $A_0$-module $A_n$ is finitely generated. It remains to prove that $A$ is finitely generated over $A_0$, and the case of a finite $L$ is, now, obvious.

In general, $L$ is a direct sum of cyclic groups, and using that $A^{L'\oplus L''}=(A^{L'})^{L''}$, we reduce the claim to the case when $L$ is cyclic. Thus, we can assume that $L=\ZZ$. It suffices to prove that $A_{\ge 0}=\oplus_{n\ge 0}A_n$ is finitely generated over $A_0$. Indeed, the same is then true for $A_{\le 0}=\oplus_{n\le 0}A_n$ by the symmetry, and we win.

Set $A_+=\oplus_{n>0}A_n$. The homogeneous ideal $A_+A$ is finitely generated, so we can choose its homogeneous generators $f_1\.f_l\in A_+$. Note that $n_i=\deg(f_i)>0$. Set $n=\max_i n_i$ and let $C\subseteq A_{\ge 0}$ be the $A_0$-subalgebra generated by $A_0\.A_n$. We claim that $C=A_{\ge 0}$, so the latter is finitely generated over $A_0$. Indeed, by induction on $m>n$ we can assume that $A_0\.A_{m-1}$ lie in $C$. Any $g\in A_m$ can be represented as $\sum g_if_i$, and the equality is preserved when we remove from each $g_i$ the components of degree different from $m-n_i$. But then $g_i\in A_{m-n_i}\subset C$ by the induction assumption, and so $g\in C$.
\end{proof}

\subsection{$L$-local rings}\label{Sec:local}
In this section we study $L$-local rings that play the role of local rings among $L$-graded rings.

\begin{remark}
In fact, one can develop an $L$-graded analogue of commutative algebra (and algebraic geometry) which goes rather far. See \cite[\S1]{Temkin-local-properties}, where graded versions of fields, local rings, fields of fractions, prime ideals, spectra, valuation rings, etc., were introduced. Many formulations and arguments are extended to the graded case just by replacing ``elements" (of a ring or a module) with ``homogeneous elements".
\end{remark}

\subsubsection{Maximal homogeneous ideals}
By a {\em maximal homogeneous ideal} of an $L$-graded ring $A$ we mean any homogeneous ideal $m\subsetneq A$ such that no homogeneous ideal $n$ satisfies $m\subsetneq n\subsetneq A$. Note that $m$ does not have to be a maximal ideal of $A$.

\subsubsection{Graded fields}
An $L$-graded ring $k$ is called an {\em $L$-graded field} if $0$ is the only proper homogeneous ideal of $A$. Equivalently, any non-zero homogeneous element of $A$ is invertible. Note that $m\subseteq A$ is a maximal homogeneous ideal if and only if $A/m$ is a graded field. Graded fields are analogues of fields in the category of graded rings. In particular, it is easy to see that any graded module over $k$ is a free $k$-module, see \cite[Lemma 1.2]{Temkin-local-properties}.

\subsubsection{$L$-local rings}
An $L$-graded ring $A$ that possesses a single maximal homogeneous ideal $m$ will be called {\em $L$-local},\index{Llocal ring@$L$-local ring} and we will often use the notation $(A,m)$. A homogeneous homomorphism $\phi\:A\to B$ of $L$-local rings is called {\em $L$-local} if it takes the maximal homogeneous ideal of $A$ to that of $B$. Here are a few other ways to characterize $L$-local rings.

\begin{lemma}\label{Lem:Llocal}
For an $L$-graded ring $A$ the following conditions are equivalent:

(i) $A$ is $L$-local.

(ii) The ring $A_0$ is local.

(iii) The action of $\bfD_L$ on $\Spec(A)$ possesses a single closed orbit.
\end{lemma}
\begin{proof}
Homogeneous ideals of $A$ correspond to closed $\bfD_L$-equivariant subsets of $\Spec(A)$, hence (i) is equivalent to (iii). On the other hand, each fiber of the quotient map $\Spec(A)\to\Spec(A_0)$ contains a single closed orbit hence (ii) is equivalent to (iii).
\end{proof}

\begin{example}\label{localexam}
A homogeneous ideal $p\subsetneq A$ is called {\em $L$-prime} in \cite[\S1]{Temkin-local-properties} if for any two homogeneous elements $x,y$ with $xy\in p$ at least one of them lies in $p$. Inverting all homogeneous elements in $A\setminus p$ one obtains an $L$-local ring $A_{p,L}$ that we call the {\em homogeneous localization} of $A$ at $p$. Even if $p$ is prime in the usual sense, $A_{p,L}$ is usually smaller than the usual localization $A_p$.
\end{example}

\subsubsection{Homogeneous nilpotent radical}
By the {\em homogeneous nilpotent radical} we mean the ideal generated by all homogeneous nilpotent elements. The following result is proved precisely as its classical ungraded analogue.

\begin{lemma}\label{homradlem}
Let $A$ be an $L$-graded ring. Then the homogeneous nilpotent radical of $A$ coincides with the intersection of all $L$-prime ideals of $A$.
\end{lemma}

\subsubsection{Graded Nakayama's lemma}

\begin{proposition}\label{nakayamaprop}
Let $(A,m)$ be an $L$-local ring and let $M$ be a finitely generated $L$-graded $A$-module. Then $mM=M$ if and only if $M=0$.
\end{proposition}
\begin{proof}
Note that $M=0$ if and only if its support is empty. By Nakayama's lemma, the latter consists of all points $x\in\Spec(A)$ such that $M(x)=M\otimes_Ak(x)\neq 0$. Since the support is $\bfD_L$-equivariant, it is given by a homogeneous ideal, and hence it is either empty or contains $V(m)$. It follows that $M=0$ if and only if $M(x)=0$ for any $x\in V(m)$. The latter is equivalent to the vanishing of $M/mM$.
\end{proof}

As in the usual situation, Nakayama's lemma has the following immediate corollary.

\begin{corollary}\label{nakayamacor}
Let $(A,m)$ be an $L$-local ring with residue graded field $k=A/m$, and let $M$ a finitely generated $L$-graded $A$-module. Then,

(i) A homogeneous homomorphism of $L$-graded $A$-modules $\phi\:N\to M$ is surjective if and only if $\phi\otimes_A k$ is surjective.

(ii) Homogeneous elements $m_1\.m_l$ generate $M$ if and only if their images generate $M/mM$.

(iii) The minimal cardinality of a set of homogeneous generators of $M$ equals to the rank of the free $k$-module $M/mM$.
\end{corollary}

\subsubsection{Equivariant Cartier divisors}
As a corollary of the graded Naka\-yama's lemma we can give the following characterization of equivariant divisors that will be used in \cite{AT1} when toroidal actions are studied. (The finite presentation assumption is only essential in the non-noetherian case.)

\begin{proposition}\label{Cartierprop}
Assume that $(A,m)$ is an $L$-local integral domain and $D\subset\Spec(A)$ is an equivariant finitely presented closed subscheme. Let $x$ be an arbitrary point of $V(m)$ and let $X_x=\Spec(\cO_{X,x})$ be the localization at $x$. Then the following conditions are equivalent:

(i) $D=V(f)$ for a homogeneous element $f\in A$,

(ii) $D$ is a Cartier divisor in $X$,

(iii) $D_x=D\times_XX_x$ is a Cartier divisor in $X_x$.
\end{proposition}
\begin{proof}
The only implication that requires a proof is (iii)$\implies$(i). By our assumptions, $D=V(I)$ for a finitely generated homogeneous ideal $I$, so $I/mI$ is a free $A/m$-module of a finite rank $d$. Since $D_x$ is Cartier, the $k(x)$-vector space $I\otimes_A k(x)=(I/mI)\otimes_A k(x)$ is one-dimensional, and we obtain that $d=1$. Then $I$ is generated by a single homogeneous element $f$ by Corollary~\ref{nakayamacor}.
\end{proof}

\subsubsection{Regular parameters}
Let $(A,m)$ be an $L$-local ring. Then $O=\Spec(A/m)$ is the closed orbit of $\Spec(A)$ and hence $A/m=k[L']$, where $k=A_0/m_0$ is the residue field of $A_0$ and $L'\subseteq L$ is a subgroup.

\begin{lemma}\label{paramlem}
Keep the above notation and assume that $A$ is a regular ring and the torsion degree of $L'$ is invertible in $k$. Let $n$ be the codimension of $O$ in $\Spec(A)$, then there exist homogeneous elements $t_1\.t_n\in A$ that generate $m$.
\end{lemma}
\begin{proof}
By Nakayama's lemma~\ref{nakayamacor}(iii), we should only check that the rank of the free $A/m$-module $m/m^2$ is $n$. Note that $m/m^2$ defines the conormal sheaf to $O$. But $O$ is regular by our assumption on the torsion of $L'$, hence the rank of the conormal sheaf is $n$.
\end{proof}

\subsection{Strictly $L$-local rings}

\subsubsection{The definition}

An $L$-local ring $(A,m)$ is called {\em strictly $L$-local}\index{strictly $L$-local ring} if $m$ contains any $A_n$ with $n\neq 0$. Here are a few natural ways to reformulate this:

\begin{lemma}\label{strictlylocallem}
Let $(A,m)$ be an $L$-local ring and let $m_0=m\cap A_0$. Then the following conditions are equivalent:

(i) $A$ is strictly $L$-local.

(ii) the closed orbit of $\bfD_L$ on $\Spec(A)$ is a point.

(iii) $A/m=A_0/m_0$.

(iv) $m=m_0\oplus(\oplus_{0\neq n\in L}A_n)$.
\end{lemma}

\subsubsection{Regularity and coinvariants}

\begin{lemma}[\protect{\cite[Corollary of Theorem 5.4]{Fogarty}}]\label{coinvlem}
If a strictly $L$-local ring $(A,m)$ is regular then the ring of coinvariants $A_L$ is regular.
\end{lemma}

\begin{proof}
Choose $t_1\.t_l$ as in Lemma~\ref{paramlem}. Since $A/m$ is a field, they form a regular family of homogeneous parameters. First, consider the case when all $t_i$ are of degree zero. We claim that $A=A_0$, and so $A_L=A$ is regular. Indeed, we have that $m=m_0A$, where $m_0=m\cap A_0$. Choose $0\neq n\in L$. Since $A_n\subset m$, we obtain that $m_0A_n=A_n$. But $A_n$ is a finitely generated $A_0$-module by Proposition~\ref{noetherprop}, and so $A_n=0$ by Nakayama's lemma.

Assume now that the degrees are arbitrary, and reorder $t_i$ so that $t_1\.t_q$ are the only elements of degree 0. Then $A'=A/(t_{q+1}\.t_l)$ is regular and we have that $A'_L=A_L$. It remains to observe that the images of $t_1\.t_q$ form a regular family of parameters of $A'$, and so $A'_L=A'$ by the above case.
\end{proof}

\subsubsection{Completion}
By an {\em $L$-complete local ring}\index{Lcomplete local ring@$L$-complete local ring} we mean a complete local ring $(A,m)$ provided with a {\em formal $L$-grading} $A=\prod_{n\in L}A_n$ such that $A_n\subset m$ for each $n\neq 0$. In particular, $A$ is strictly $L$-local in the formal sense and its residue field is trivially graded. Homogeneous homomorphisms of $L$-complete local rings are defined in the obvious way. The main motivation for considering this class of rings is the following result.

\begin{proposition}\label{Prop:complete}
Assume that $L$ is a finitely generated abelian group and $(A,m)$ is a noetherian strictly $L$-local ring. Set $m_0=m\cap A_0$ and for each $n\in L$ let $\hatA_n$ denote the $m_0$-adic completion of the $A_0$-module $A_n$. Then, the $m$-adic completion of $A$ is isomorphic to $\prod_{n\in L}\hatA_n$; in particular, it is an $L$-complete local ring.
\end{proposition}
\begin{proof}
Write $m_0^A = m_0A$. We need to relate the $m$-adic and $m_0^A$-adic topologies on $A$. We have the following:

\begin{lemma}\label{Lem:combinatorics}
\begin{enumerate}
\item There is a constant $b>0$ and, for every $n\in L$ a constant $r(n)>0$, such that the following holds. Let $a\in A_n \cap m^N$. Then if $N \geq bq + r(n)$ then $a \in  (m_0^A)^q$.
\item For each $N$ the set $\{n\in L | A_n\not\subset m^N\}$ is finite.
\end{enumerate}
\end{lemma}

\begin{proof}[Proof of the lemma]
(1) By Proposition~\ref{noetherprop}, $A$ is finitely generated over $A_0$, hence we can choose homogeneous $A_0$-generators $f_1,\ldots,f_k\in A\setminus A_0$. Denote their degrees by $n_1,\ldots,n_k\in L\setminus\{0\}$ and consider the monoid homomorphism $\phi\:\NN^k \to L$ sending the $i$-th generator to $n_i$. An element of $A_n$ is a polynomial in the $f_i$ whose monomials have exponents in $\phi^{-1} (n)$. Any monomial in $\phi^{-1} (0)$ lies in $m_0$, and the claim is proven if we find  $b>0$ and $r(n)>0$ and show that each monomial appearing in $a\in A_n \cap m^N$ with $N\geq bq+r(n)$ factors at least $q$ monomials in $m_0$. Writing $|(l_1,\ldots,l_k)| = \sum l_i$, we need to find $b>0$ and $r(n)>0$ and  show that each monomial $f_1^{l_1}\cdots f_k^{l_k}$ of degree $|(l_1,\ldots,l_k)|=N\geq bq+r(n)$ in $f_i$ such that $\phi(l_1,\ldots,l_k) = n$  factors into at least $q$ monomials lying in $m_0$. This is a combinatorial question on lattices.

Consider first the case $n=0$. The monoid $\phi^{-1} (0)$ is finitely generated, $\phi^{-1}(0)  = \<g_1,\ldots, g_s\>$.  Write $b = \max(|g_i|)$. If $(l_1,\ldots,l_k) \in \phi^{-1}(0)$ then it has an integer expression $(l_1,\ldots,l_k) = \sum q_i g_i$, and $q=\sum q_i \geq  N/\max(|g_i|)$. So, $f_1^{l_1}\cdots f_k^{l_k}$ is the product of at least $q$ elements of the form $f^g$ in $m_0$. This gives the result in this case.

Consider the general case. Since the monoid $\NN^k$ is noetherian, the ideal generated by $\phi^{-1} (n)$ has finitely many generators $h_1,\ldots,h_t\in \phi^{-1} (n)$. It follows that $\phi^{-1} (n) = \cup_i \left(h_i + \phi^{-1} (0)\right)$. Write $r = \max(|h_i|)$. Then any
 $(l_1,\ldots,l_k) \in \phi^{-1} (n)$ of degree $N\geq bq+r$ can be written as $h_i + t$, with $t\in \phi^{-1} (0)$ of degree $\geq bq$, hence $t$ is the sum of at least $q$ elements of $\phi^{-1} (0)$, which gives the general case.

(2) Let $a\in A_n\subset A$ with $a\not\in m^N$. Then any expression $a = \sum c_j f_1^{l_{j1}} \cdots f_k^{l_{jk}}$ has at least one monomial of degree $|(l_{j1},\ldots,l_{jk})| < N$.  But the set of $(l_{1},\ldots,l_{k})$ of degree $<N$ is finite, therefore the set of  $n = \phi(l_{1},\ldots,l_{k})$ with $|(l_{1},\ldots,l_{k})|<N$ is finite, as required.
\end{proof}
We continue to prove Proposition~\ref{Prop:complete}. Note that $\hatA_n = \varprojlim_N A_n / (m^N\cap A_n)$ since $m_0^NA_n \subset (m^N\cap A_n) \subset m_0^{q(N)}A_n$ by Lemma~\ref{Lem:combinatorics}(1). Also, claim (2) of the same lemma implies that $\oplus_n A_n/(m^N\cap A_n)=\prod_n A_n/(m^N\cap A_n)$. Using that products are compatible with limits we obtain that
\begin{eqnarray*}
\hatA=\varprojlim_N A/m^N=\varprojlim_N\oplus_{n\in L}A_n/(m^N\cap A_n)=\varprojlim_N\prod_{n\in L}A_n/(m^N\cap A_n)=\\\prod_{n\in L}\varprojlim_NA_n/(m^N\cap A_n)=\prod_{n\in L}\hatA_n,
\end{eqnarray*}
as required.
\end{proof}

\subsection{Strongly homogeneous homomorphisms}\label{strhomsec}
We say that a homogeneous homomorphism $\phi\:A\to B$ of $L$-graded rings (resp. $L$-complete local rings) is {\em strongly homogeneous} if $A\otimes_{A_0}B_0=B$ (resp. $A\wtimes_{A_0}B_0=B$) as $L$-graded rings (resp. $L$-complete local rings).\index{strongly homogeneous homomorphism} Similarly, if $P$ is a property of homomorphisms stable under base changes, e.g. formally smooth, then we say that $\phi$ is strongly $P$ if it is strongly homogeneous and $\phi_0\:A_0\to B_0$ satisfies $P$. Our next aim is to establish a criterion of strong homogeneity in one of the following two cases:
\begin{itemize}
\item[(1)] The local case: $\phi$ is a local homomorphism of noetherian strictly $L$-local rings.
\item[(2)] The formal case: $\phi$ is a local homomorphism of noetherian $L$-complete local rings.
\end{itemize}

\subsubsection{Reduction to the formal case}
First, let us reduce the problem to the formal case.

\begin{lemma}\label{localtoformal}
Assume that $\phi\:A\to B$ is a homogeneous local homomorphism of noetherian strictly $L$-local rings. Then $\phi$ is strongly homogeneous if and only if its completion $\hatphi\:\hatA\to\hatB$ is strongly homogeneous.
\end{lemma}
\begin{proof}
Clearly, $\phi$ is strongly homogeneous if and only if $f_n\:A_n\otimes_{A_0}B_0\to B_n$ is an isomorphism for any $n\in L$, and it follows from Proposition~\ref{Prop:complete} that $\hatphi$ is strongly homogeneous if and only if $g_n\:\hatA_n\wtimes_{\hatA_0}\hatB_0\to\hatB_n$ is an isomorphism for any $n\in L$. Since $B_n$ is finitely generated over $B_0$ we have that $\hatB_n=B_n\otimes_{B_0}\hatB_0$, and since $A_n$ is finitely generated over $A_0$ we have that $\hatA_n\wtimes_{\hatA_0}\hatB_0=A_n\otimes_{A_0}\hatB_0$. Thus, $g_n\:A_n\otimes_{A_0}\hatB_0\to B_n\otimes_{B_0}\hatB_0$ is the base change of $f_n$ with respect to the faithfully flat homomorphism $B_0\to\hatB_0$, and hence $f_n$ is an isomorphism if and only $g_n$ is an isomorphism. The lemma follows.
\end{proof}

\subsubsection{The fiber}\label{Sec:fib}
Let $\phi$ be as in cases (1) or (2). Set $m=m_A$ and $k=A/m$ and consider the {\em fiber} $\ophi\:k\to\Lam=B/mB$ of $\phi$. If $\phi$ is strongly homogeneous then the fiber is strongly homogeneous too, and since $k$ is trivially graded we obtain that $\Lam$ is trivially graded too. The geometric meaning of the latter condition is that the action of $\bfD_L$ on the fiber $\Spec(\Lambda)$ is trivial.

\begin{lemma}\label{fiblem}
Let $k$ be a trivially graded field, $\Lam$ a noetherian local $k$-algebra with residue field $l=\Lam/m_\Lam$ and an $L$-grading making the structure homomorphism $\psi\:k\to\Lam$ homogeneous. Provide the $l$-vector spaces $\Omega_{\Lam/k}\otimes_\Lam l$ and $m_\Lam/m_\Lam^2$ with the induced $L$-grading. Then the following conditions are equivalent:

(i) $\Lam$ is trivially graded.

(ii) The field $l$ and the space $m_\Lam/m_\Lam^2$  are trivially graded.

(iii) The field $l$ and the space $\Omega_{\Lam/k}\otimes_\Lam l$ are trivially graded.
\end{lemma}
\begin{proof}
Clearly, (i) implies (ii). Conversely, if (ii) holds then each vector space $m_\Lam^n/m_\Lam^{n+1}$ is trivially graded. Since $\cap_n m_\Lam^n=0$ by the noetherian assumption, it follows that $\Lam$ is trivially graded.

To prove equivalence of (ii) and (iii) we observe that the homomorphisms $k\to\Lam\to l$ induce an exact triangle of cotangent complexes  - the transitivity triangle \cite[II.2.1.2.1]{Illusie-cotangent} - whose associated exact sequence of homologies contains
$$\Upsilon_{l/k}\to m_\Lam/m_\Lam^2\to\Omega_{\Lam/k}\otimes_\Lam l\to\Omega_{l/k}.$$
Both in (ii) and (iii), $l$ is trivially graded and hence the end terms $\Upsilon_{l/k}$ and $\Omega_{l/k}$ are trivially graded. Thus, $m_\Lam/m_\Lam^2$ is trivially graded if and only if $\Omega_{\Lam/k}$ is trivially graded, as needed.\end{proof}

\begin{remark}
The assumption on the grading of $l$ is automatic when $\Lam$ is strictly $L$-local (as will be in all our applications), but it cannot be omitted in general. For example, consider the case when $k=\RR$ and $\Lam=l=\CC=\RR\oplus i\RR$ as a $\ZZ/2\ZZ$-graded field.
\end{remark}

\subsubsection{Lifting a formal grading}
The following lemma will be our main tool in studying homomorphism with trivially graded fiber.

\begin{lemma}\label{liftlem}
Let $\phi\:A\to B$ be a homogeneous local homomorphism of $L$-complete noetherian local rings with a trivially graded fiber $k\to\oB=B/m_AB$. Assume that $g\:A\to D$ and $h\:D\onto B$ form a formally smooth factorization of $\phi$ (\ref{factorizablesec}). Then one can provide $D$ with a formal grading such that the following conditions hold:

(i) Both $g$ and $h$ are homogeneous homomorphisms.

(ii) The fiber $\oD=D/m_AD$ is trivially graded.

(iii) $g$ is a graded base change of the trivially graded homomorphism $C(k)\to C(\oD)$ introduced in \ref{mixedcase}.
\end{lemma}
\begin{proof}
Since $\oB$ is trivially graded we have $\oB=B_0/m_0B_0$, where $m_0$ is the trivially graded component of $m_A$. Fix a ring of coefficients $i\:C(k)\to A_0$, so that all rings we consider become $C(k)$-algebras. In particular we obtain homomorphisms $i_{B_0}:C(k) \to B_0$ and $i_D: C(K) \to D$ which we utilize below.

Recall that the homomorphism $C(\og)\:C(k)\to C(\oD)$ discussed in \ref{mixedcase} is formally smooth; considering the composed homomorphism $C(\oD)\to\oD\to\oB$ and the diagram
$$\xymatrix{ C(k) \ar[r]^{i_{B_0}}\ar[d]_{C(\og)} & B_0 \ar[d] \\
C(\oD)\ar[r] \ar@{.>}[ru] & \oB}$$
we can lift  $C(\oD)\to\oB$ to a $C(k)$-homomorphism $C(\oD)\to B_0$.

We claim that the composed homomorphism $\lambda\:C(\oD)\to B_0\to B$ lifts to a homomorphism $C(\oD)\to D$ such that the composition $C(\oD)\to D\to\oD$ is the canonical projection $p\:C(\oD)\to\oD$. Indeed, we have  natural homomorphisms $(\lambda,p)\:C(\oD)\to B\times_\oB\oD$ and $D\to B\times_\oB\oD$, giving a {\em commutative} diagram
$$\xymatrix{ C(k) \ar[r]^{i_{D}}\ar[d]_{C(\og)} & D \ar[d] \\
C(\oD)\ar[r]_{(\lambda,p)} \ar@{.>}[ru] &  B\times_\oB\oD.}$$
By the Chinese remainder theorem, the homomorphism $D\to B\times_\oB\oD$ is surjective, and using the formal smoothness of $C(\og)$ again we can lift $(\lambda,p)$ to a homomorphism $C(\oD)\to D$.

At this stage Theorem~\ref{Th:splitfsmoothmixed}(ii), applied to the formally smooth homomorphism $g: A \to D$, asserts that $D=A\widehat\otimes_{C(k)}C(\oD)$, making $g$ the base change of  $C(k)\to C(\oD)$. Provide $C(\oD)$ with the trivial grading. This induces the trivial grading on $\oD$, giving (ii). The formal grading $A = \prod_{n\in L} A_n$ provides a formal grading $D = \prod_{n\in L} A_n\widehat\otimes_{C(k)}C(\oD)$ making $g$ a graded base change, giving (iii); in particular $g$ is homogeneous. The homomorphism $C(\oD)\to B$ is homogeneous since it factors through $B_0$. This implies that $A_n\widehat\otimes_{C(k)}C(\oD) \to B$ factors through $A_nB_0 \subset B_n$, so $h$ is homogenous, completing (i) as needed.
\end{proof}

\subsubsection{Strong formal smoothness}
Already the case $D=B$ of the above lemma allows to completely describe strongly formally smooth homomorphism (as defined in Section \ref{strhomsec}).

\begin{corollary}\label{liftcor}
Assume that $\phi\:A\to B$ is a homogeneous local homomorphism of $L$-complete noetherian local rings. Then,

(i) $\phi$ is strongly formally smooth if and only if it is formally smooth and has a trivially graded fiber.

(ii) If $\phi$ is formally smooth then the following conditions are equivalent: (a) $\phi$ has a trivially graded fiber, (b) $\phi$ is strongly homogeneous, (c) $\phi$ is strongly formally smooth.
\end{corollary}
\begin{proof}
The forward direction in (i) follows from the discussion in \ref{Sec:fib}, so we assume $\phi$ is formally smooth with trivially graded fiber, and prove that $\phi$ is strongly homogeneous and $\phi_0$ is formally smooth. Applying Lemma~\ref{liftlem} with $D=B$ we obtain that $B=A\widehat\otimes_{C(k)}C(\oB)$, where $C(k)$ and $C(\oB)$ are trivially graded. In particular, $B_0=A_0\widehat\otimes_{C(k)}C(\oB)$ and hence $\phi$ is strongly homogeneous. Moreover, $\phi_0$ is a base change of the formally smooth homomorphism $C(k)\to C(\oB)$, hence $\phi_0$ is formally smooth and we obtain (i).

To prove (ii) we note that the implications (c)$\implies$(b)$\implies$(a) are obvious, and the implication (a)$\implies$(c) follows from (i).
\end{proof}

\subsubsection{Strong homogeneity and the cotangent complex}
Now we will study general factorizable homomorphisms, see \ref{factorizablesec}. In this case, one should also control the grading of the kernel of a factorization $D\onto B$. Naturally, this is related to the first homology of the cotangent complex $\LL_{B/A}$.

\begin{lemma}\label{idealgrading}
Assume that $g\:A\to D$ is a formally smooth homogeneous local homomorphism of $L$-complete noetherian local rings with a trivially graded fiber $k\to\oD$, let $I\subset D$ be a homogeneous ideal, $B=D/I$ and $l$ the residue field of $B$. Then the following conditions are equivalent:

(i) The composition $A\to B$ is strongly homogeneous.

(ii) $I$ is generated by elements of degree zero.

(iii) $(I/I^2)\otimes_Bl$ is trivially graded.

(iv) $H_1(\LL_{B/A}\otimes^{\rm L}_Bl)$ is trivially graded.
\end{lemma}
\begin{proof}
(i)$\Longleftrightarrow$(ii) The homomorphism $A\to D$ is strongly homogeneous by Corollary~\ref{liftcor}(i), hence $A\to B$ is strongly homogeneous if and only if $D\to B$ is strongly homogeneous. The latter happens if and only if $D/I=B=D\otimes_{D_0}B_0=D/I_0B$, where $I_0$ is the trivially graded part of $I$, i.e. if and only if (ii) holds.

(ii)$\Longleftrightarrow$(iii) The direct implication is clear. Conversely, assume that (iii) holds, then the $B$-module $I/I^2$ is generated by elements of degree 0 by the graded Naka\-yama's Lemma, see Corollary~\ref{nakayamacor}(ii). Since elements of $I$ generating  $I/I^2$ generate $I$ by the usual Nakayama's Lemma, we obtain (ii).

(iii)$\Longleftrightarrow$(iv) Consider the exact triangle $\Delta$ obtained from the transitivity triangle
$$\LL_{D/A}\otimes^{\rm L}_DB\to\LL_{B/A}\to\LL_{B/D}\to\LL_{D/A}\otimes^{\rm L}_DB[1]$$ by applying $\cdot\otimes^{\rm L}_B l$.
By the formal smoothness of $A\to D$, the complex $\LL_{D/A}\otimes^{\rm L}_Dl$ is quasi-isomorphic to $\Omega_{D/A}\otimes_D l$. By \cite[III.1.2.8.1]{Illusie-cotangent} we have  $H_0(\LL_{B/D})=0$ and $H_1(\LL_{B/D})=I/I^2$, therefore $H_1(\LL_{B/D}\otimes_B^{\rm L} l)=(I/I^2)\otimes_Bl$. It follows that the exact sequence of homologies associated with $\Delta$ contains the sequence $$0\to H_1(\LL_{B/A}\otimes^{\rm L}_Bl)\to(I/I^2)\otimes_Bl\to\Omega_{D/A}\otimes_D l.$$ By our assumption on $A\to D$, the  term $\Omega_{D/A}\otimes_D k$ is trivially graded, hence $(I/I^2)\otimes_Bl$ is  trivially graded if and only if $H_1(\LL_{B/A}\otimes^{\rm L}_Bl)$ is trivially graded.
\end{proof}

\subsubsection{The main result: formal factorizable case}
We summarize what we have done so far in the factorizable case. This will be later generalized to arbitrary homomorphisms, see Theorem~\ref{mainformalth}, but our result here is slightly more precise since we use only formally smooth factorizations in (ii).

\begin{theorem}\label{formalcaseth}
Assume that $\phi\:A\to B$ is a formally factorizable homogeneous local homomorphism of $L$-complete noetherian local rings. Then the following conditions are equivalent:

(i) $\phi$ is strongly homogeneous.

(ii) $\phi$ factors into a composition of homogeneous homomorphisms $A\to D\to D/I=B$, where $D$ is $L$-complete, $A\to D$ is formally smooth with a trivially graded fiber and $I$ is generated by elements of degree zero.

(iii) The $l$-vector spaces $H_0(\LL_{B/A}\otimes^{\rm L}_Bl)=\Omega_{B/A}\otimes_Bl$ and $H_1(\LL_{B/A}\otimes^{\rm L}_Bl)$ are trivially graded.
\end{theorem}
\begin{proof}
By Lemma~\ref{fiblem} the homomorphism $\phi$ has a trivially graded fiber if and only if $\Omega_{B/A}\otimes_Bl$ is trivially graded. This means that any of the  conditions (i), (ii) or (iii) implies that both  $\phi$ has a trivially graded fiber and $\Omega_{B/A}\otimes_Bl$ is trivially graded.

Fix a factorization $A\to D\onto B$ with a formally smooth $A\to D$. Since $\phi$ has trivially graded fiber we may apply Lemma~\ref{liftlem}: in case (i) or (iii) holds we may  grade $D$ so that $A \to D$ and  $D \to B$  become homogeneous and with a trivially graded fiber; this holds by assumption in case (ii).

It follows that under any of the assumptions (i), (ii) or (iii) the setup of  Lemma~\ref{idealgrading} holds. Furthermore, (i) is equivalent to  Assumption (i) of Lemma~\ref{idealgrading}, (ii) is equivalent to  Assumption (ii) of Lemma~\ref{idealgrading},  and (iii) is equivalent to  Assumption (iv) of Lemma~\ref{idealgrading}.  The equivalence of (i), (ii) and (iii) now follows from the equivalence in Lemma~\ref{idealgrading}.
\end{proof}

\begin{remark}\label{seprem}
Let $\phi\:A\to B$ be a local homomorphism of $L$-complete noetherian local rings and assume that the extension of the residue fields $l/k$ is separable (in particular, $\phi$ is factorizable). In this case, one can describe $\phi$ very explicitly.

Fix a ring of coefficients $C(k)\to A$. If $\phi$ is formally smooth then by Remark~\ref{formalrem} there exists an isomorphism of $A$-algebras $A\wtimes_{C(k)}C(l)\llbracket t_1\.t_n\rrbracket\toisom B$. It follows that in general $\phi$ is strongly homogeneous if and only if there exists a homogeneous isomorphism of $A$-algebras $$\left(A\wtimes_{C(k)}C(l)\llbracket t_1\.t_n\rrbracket\right)/(f_1\.f_m)\toisom B,$$ where $t_i,f_j$ are of degree zero.
\end{remark}

\subsection{Elimination of the formal factorization assumption}
The aim of this section is to extend main results of Section \ref{strhomsec} to non-factorizable homomorphisms. This only requires to replace formally smooth factorizations by so-called Cohen factorizations. The arguments are very similar, so we will mainly indicate the modifications one has to make.

\subsubsection{Cohen factorization}
Let $\phi\:A\to B$ be a homomorphism of noetherian complete local rings. A {\em Cohen factorization} of $\phi$ is a factorization of the form $A\stackrel{f}\to D\stackrel{g}\onto B$, where $D$ is a noetherian complete local ring, $g$ is surjective, $f$ is flat, and the ring $D/m_AD$ is regular.

\begin{rem}
(i) Cohen factorizations were introduced by Avramov, Foxby and Herzog in \cite{AFH} to study local homomorphisms between noetherian complete local rings. In particular, they showed that such a factorization always exists. Later, this notion was exploited by Avramov in \cite{Avramov-lci} to establish foundational properties of general lci morphisms.

(ii) For the sake of comparison with the formally smooth factorizations recall that $f$ is formally smooth if it is flat and $D/m_AD$ is geometrically regular over $A/m_A$. Flat morphisms with regular fibers are sometimes called weakly regular, but they are not especially useful. So it may be surprising that
Cohen factorization do provide a useful tool.
\end{rem}

\subsubsection{Graded Cohen factorization}
The following result will serve as a replacement of Lemma~\ref{liftlem}.

\begin{lemma}\label{cohenfactorlem}
Let $\phi\:A\to B$ be a homogeneous local homomorphism of $L$-complete noetherian local rings with a trivially graded fiber $k\to\oB=B/m_AB$. Then there exists a Cohen factorization $A\stackrel{g}\to D\stackrel{h}\onto B$ such that the following conditions hold:

(i) Both $g$ and $h$ are homogeneous homomorphisms.

(ii) $g$ is a graded formal base change of a homomorphism $\psi\:C(k)\to C(l)\llbracket y_1\.y_n\rrbracket$, where $l=B/m_B$ and the gradings are trivial. In particular, the fiber $\oD=D/m_AD$ is trivially graded.
\end{lemma}
We note that, unlike the situation in Remark~\ref{seprem}, the homomorphism $\psi$ does not take $C(k)$ to $C(l)$ in general.
\begin{proof}
The proof is a graded variation on the proof of \cite[Theorem 1.1]{AFH}, and it is also close to the proof of Lemma~\ref{liftlem}, so we only describe the construction. The rings $A$ and $B$ have the same residue fields as $A_0$ and $B_0$, respectively, so we can fix structure homomorphisms $C(k)\to A_0$ and $C(l)\to B_0$. Choose a surjective homomorphism $C(l)\llbracket y_1\.y_n\rrbracket\onto B_0$ that takes $y_i$ to a family of generators of $m_{B_0}$, then the composed homomorphism $C(k)\to A_0\to B_0$ lifts to a homomorphism $\psi\:C(k)\to C(l)\llbracket y_1\.y_n\rrbracket$. We provide $\psi$ with the trivial grading and define $g$ to be the graded formal base change $A\to D=A\wtimes_{C(k)}C(l)\llbracket y_1\.y_n\rrbracket$. We claim that $g$ satisfies (ii). Indeed, $g$ is flat because $\psi$ and the completion homomorphism $A\otimes_{C(k)}C(l)\llbracket y_1\.y_n\rrbracket\to D$ are flat, and $D/m_AD=l\llbracket y_1\.y_n\rrbracket$ is regular.

It remains to construct $h$. The graded homomorphisms $C(l)\llbracket y_1\.y_n\rrbracket\onto B_0\to B$ and $\phi$ induce a homomorphism $h\:D\to B$, so we should only check that $h$ is onto. Since $D/m_D=l$ we should only prove that $h(D)$ contains a set of generators of $m_B$. Clearly, $h(D)$ contains $m_AB$ and $m_{B_0}$, and it remains to note that since $B/m_AB$ is trivially graded, $m_B=m_AB+m_{B_0}$.
\end{proof}

\subsubsection{The cotangent complex}
Next, we adjust Lemma \ref{idealgrading} to Cohen factorizations.

\begin{lemma}\label{cohenidealgrading}
Assume that $g\:A\to D$ is a graded formal base change of a trivially graded homomorphism $\psi\:C(k)\to C(l)\llbracket y_1\.y_n\rrbracket$, where $k=A/m_A$ and $l=D/m_D$. If $I\subset D$ is a homogeneous ideal and $B=D/I$ then the following conditions are equivalent:

(i) The composition $A\to B$ is strongly homogeneous.

(ii) $I$ is generated by elements of degree zero.

(iii) $(I/I^2)\otimes_Bl$ is trivially graded.

(iv) $H_1(\LL_{B/A}\otimes^{\rm L}_Bl)$ is trivially graded.
\end{lemma}
\begin{proof}
The proof is a copy of the proof of Lemma \ref{idealgrading} with the only difference that in the exact sequence
$$H_1(\LL_{D/A}\otimes^{\rm L}_Dl)\to H_1(\LL_{B/A}\otimes^{\rm L}_Bl)\to(I/I^2)\otimes_Bl\to\Omega_{D/A}\otimes_Dl$$ the first term can be non-zero. However, we claim that it is trivially graded, and hence it is still true that $(I/I^2)\otimes_Bl$ is trivially graded if and only if $H_1(\LL_{B/A}\otimes^{\rm L}_Bl)$ is trivially graded.

It remains to study $H_1(\LL_{D/A}\otimes^{\rm L}_Dl)$. Set $E=C(l)\llbracket y_1\.y_n\rrbracket$. Since $\psi$ is flat, the homomorphisms $\psi$ and $C(k)\to A$ are $\Tor$-independent and hence $\LL_{D/A}\toisom\LL_{E/C(k)}\otimes^{\rm L}_ED$ by \cite[Corollary~III.2.2.3]{Illusie-cotangent}. Thus, $\LL_{D/A}\otimes^{\rm L}_Dl=\LL_{E/C(k)}\otimes^{\rm L}_El$ is trivially graded since $\psi$ is trivially graded.
\end{proof}

\subsubsection{The main result: formal case}
Now, we can eliminate the formal factorization assumption from our main formal result.

\begin{theorem}\label{mainformalth}
Assume that $\phi\:A\to B$ is a homogeneous local homomorphism of $L$-complete noetherian local rings. Then the following conditions are equivalent:

(i) $\phi$ is strongly homogeneous.

(ii) $\phi$ possesses a Cohen factorization $A\stackrel{g}\to D\to D/I=B$, where $g$ a graded formal base change of a trivially graded homomorphism $\psi\:C(k)\to C(l)\llbracket y_1\.y_n\rrbracket$ and $I$ is generated by elements of degree zero.

(iii) The $l$-vector spaces $H_0(\LL_{B/A}\otimes^{\rm L}_Bl)=\Omega_{B/A}\otimes_Bl$ and $H_1(\LL_{B/A}\otimes^{\rm L}_Bl)$ are trivially graded.
\end{theorem}
\begin{proof}
The proof of Theorem \ref{formalcaseth} applies, with Lemmas~\ref{fiblem} and \ref{idealgrading} replaced by Lemmas~\ref{cohenfactorlem} and \ref{cohenidealgrading}.
\end{proof}

\subsubsection{Descent of lci}\label{Sec:lci-syntomic}
The notion of general lci morphisms between noetherian schemes was introduced by Avramov, see Definition on pages 458-459 of \cite{Avramov-lci}. This is equivalent to the following: a local homomorphism of notherian local rings $\phi\:A\to B$ is called {\em complete intersection} if its completion possesses a Cohen factorization $\hatA\to D\to\hatB$ such that the kernel of $D\to\hatB$ is generated by a regular sequence. This turns out to be independent of the factorization. A morphism of locally noetherian schemes is {\em lci}\index{lci morphism} if all its local homomorphisms are complete intersections. In particular, $\phi$ is complete intersection if and only if it is lci, and we will use the notion lci in the sequel.

Recall  that a flat lci morphism is called {\em syntomic}. Traditionally one assumes a syntomic morphism to be locally of finite presentation, but with Avramov's general notion of lci morphisms this assumption becomes redundant. Unlike general lci morphisms, syntomic morphisms are preserved by arbitrary base changes, so the notion of being strongly syntomic makes sense (as defined in Section \ref{strhomsec}).

\begin{lemma}\label{cilem}
Let $\phi\:A\to B$ be a homogeneous local homomorphism of $L$-complete noetherian local rings.

(i) Assume that $\phi$ is strongly homogeneous. If $\phi$ is lci then $\phi_0\:A_0\to B_0$ is lci. Conversely, if $\phi_0$ is lci and the homomorphisms $\phi_0$ and $A_0\into A$ are $\Tor$-independent then $\phi$ is lci.

(ii) $\phi$ is strongly syntomic if and only if it is syntomic and strongly homogeneous.
\end{lemma}
\begin{proof}
The implication (i)$\implies$(ii) is obvious, so let us prove (i). The inverse implication is clear since lci morphisms are preserved by $\Tor$-independent base changes and compositions, and the completion homomorphism $A\otimes_{A_0}B_0\to A\wtimes_{A_0}B_0=B$ is lci.

Assume now that $\phi$ is lci. Consider a Cohen factorization $A\to D\to D/I=B$ as in Theorem~\ref{mainformalth}. Then $I$ is generated by a regular sequence of length $n$ and hence any generating sequence of $I$ of size $n$ is a regular sequence. Note that $n$ is the dimension of $(I/I^2)\otimes_Dl$, where $l=D/m_D$. By Lemma~\ref{cohenidealgrading}, $(I/I^2)\otimes_Dl$ is trivially graded, hence we can find $x_1\.x_n$ of degree zero whose images generate $(I/I^2)\otimes_Dl$. By Nakayama's lemma, $x_i$ generate $I$. As we noted above, $x_i$ form a regular sequence in $D$ and hence also in $D_0$. It remains to note that $A_0\to D_0$ is a formal base change of $C(k)\to C(l)\llbracket y_1\.y_m\rrbracket$ and hence $A_0\to D_0\to B_0$ is a Cohen factorization of $\phi_0$, and the kernel of $D_0\to B_0$ coincides with $I_0$ and hence is generated by the regular sequence $x_1\.x_n$.
\end{proof}

\begin{remark}
It may happen that $\phi$ is strongly homogeneous and $\phi_0$ is lci but $\phi$ is not lci. For example, take $A_0=k\llbracket x\rrbracket$ and $B_0=A_0/(x)=k$. Extend $A_0$ to a graded algebra $A=A_0[\veps]/(\veps^2,x\veps)$, where $\veps$ is homogeneous of a non-zero degree. Finally, set $B=A\wtimes_{A_0}B_0=A/xA=k[\veps]/(\veps^2)$. Then $\phi\:A\to B$ is strongly homogeneous but not lci since $x\veps=0$ and hence $\{x\}$ is not a regular sequence.
\end{remark}

\subsection{The strictly local Luna's Fundamental Lemma}
We complete our study of strong homogeneity with a summary of  the strictly  local case. Section~\ref{Lunasection} is devoted to the {\em global} relatively affine case.

\begin{theorem}\label{localcaseth}
Assume that $\phi\:A\to B$ is a homogeneous local homomorphism of noetherian strictly $L$-local rings. Then

(i) $\phi$ is strongly homogeneous if and only if the $l$-vector spaces $H_0(\LL_{B/A}\otimes^{\rm L}_Bl)=\Omega_{B/A}\otimes_Bl$ and $H_1(\LL_{B/A}\otimes^{\rm L}_Bl)$ are trivially graded.

(ii) Assume that $\phi$ is formally smooth then the following conditions are equivalent: (a) $\Omega_{B/A}\otimes_Bl$ is trivially graded, (b) $\phi$ has a trivially graded fiber, (c) $\phi$ is strongly homogeneous, (d) $\phi$ is strongly formally smooth.

(iii) $\phi$ is strongly syntomic if and only if it is syntomic and strongly homogeneous. Moreover, if $\phi$ is strongly homogeneous then the following claims hold: (a) if $\phi$ is lci then $\phi_0$ is lci, (b) if $\phi_0$ is lci and $\Tor$-independent with the homomorphism $A_0\to A$ then $\phi$ is lci.
\end{theorem}
\begin{proof}
Recall that by Lemma~\ref{localtoformal}, $\phi$ is strongly homogeneous if and only if its completion $\hatphi\:\hatA\to\hatB$ is so. Since $H_i(\LL_{B/A}\otimes^{\rm L}_Bl)=H_i(\LL_{\hatB/\hatA}\otimes^{\rm L}_\hatB l)$ by \cite[Lemma~1]{Franco-Rodicio}, claim (i) follows from Theorem~\ref{formalcaseth}.

Parts (a) and (b) of claim (ii) are equivalent by Lemma~\ref{fiblem}. Equivalence of (b), (c) and (d) reduces to Corollary~\ref{liftcor}(ii) because $\phi_0\:A_0\to B_0$ is formally smooth if and only if its completion is formally smooth and the latter coincides with the degree-zero part of $\hatphi$ by Proposition~\ref{Prop:complete}.

Note that $\phi$ is lci or syntomic if and only if $\hatphi$ is so. Since $(\hatphi)_0$ is the completion of $\phi_0$ by Proposition~\ref{Prop:complete}, it suffices to prove (iii) for $\hatphi$, and this has already been done in Lemma \ref{cilem}.
\end{proof}

\section{Luna's fundamental lemma}\label{Lunasection}
In Section~\ref{Lunasection} we study relatively affine actions of diagonalizable groups on general noetherian schemes and extend the classical Luna's fundamental lemma to this case. To simplify the exposition we work with split groups, and indicate in Section~\ref{nonsplitsec} how the non-split case can be deduced.

\subsection{Relatively affine actions}\label{relaffinesec}

\subsubsection{The definition}
An action of $G=\bfD_L$ on a scheme $X$ is called {\em relatively affine}\index{relatively affine action} if there exists a scheme $Z$ provided with the trivial $G$-action and an affine $G$-equivariant morphism $f\:X\to Z$. In this case we define the {\em quotient} $X\sslash G={\mathbf\Spec}_Z(f_*(\cO_X)^G)$. We omit $Z$ in the notation because the quotient is categorical by the following theorem, and hence it is independent of the scheme $Z$. By definition, if $Y=X\sslash G$ is covered by affine open subschemes $Y_i$ then $X_i=Y_i\times_YX$ form an open affine equivariant covering of $X$ and $Y_i=X_i\sslash G$. Therefore, Lemma~\ref{Lem:fib} and Corollary~\ref{univcor} extend to the relative situation:

\begin{theorem}\label{relafftheor}
Assume that a scheme $X$ is provided with a relatively affine action of a diagonalizable group $G$. Then,

(i) The morphism $X\to Y=X\sslash G$ is submersive and $Y$ is the universal categorical quotient of the action.

(ii) For each $y\in Y$, the fiber $X_y$ contains a single orbit $O$ which is closed in $X_y$, and this orbit belongs to the closure of any other orbit contained in $X_y$.
\end{theorem}

In the sequel, we will only consider relatively affine actions.

\begin{remark}
The notion of a relatively affine action is not as meaningful for non-reductive groups because it does depend on $Z$. For instance, in the situation of Example~\ref{borelexam} we can take $Z$ to be either $\Spec(k)$ or $\PP_k^1$. For both choices, the relative quotient coincides with $Z$.
\end{remark}

\subsubsection{Strongly equivariant open subschemes}
Assume that $X$ is provided with a relatively affine action of $G=\bfD_L$. An open subscheme $U\into X$ is called {\em strongly equivariant} if it is the preimage of an open subscheme $V\into Y\sslash G$. Note that $V=U\sslash G$. We used a covering of $X$ by strongly equivariant schemes to prove Theorem~\ref{relafftheor}. In general, one can  use strongly equivariant open subschemes to work locally on $X\sslash G$ without describing the quotients explicitly.

\begin{lemma}\label{streqlem}
Assume that a diagonalizable group $G=\bfD_L$ acts on a scheme $X$ and $X_i$ form a covering of $X$ by open equivariant subschemes. If the action on each $X_i$ is relatively affine and each intersection $X_{ij}$ is strongly equivariant in both $X_i$ and $X_j$, then the action on $X$ is relatively affine and each $X_i$ is strongly equivariant in $X$.
\end{lemma}
\begin{proof}
By definition, the schemes $Y_i=X_i\sslash G$ glue along their open subschemes $Y_{ij}=X_{ij}\sslash G$ to  form a scheme $Y$. Hence the quotient morphisms $X_i\to Y_i$ glue to an affine quotient morphism $X\to Y$.
\end{proof}

\begin{example}\label{streqexam}
If $X=\Spec(A)$, where $A$ is $L$-graded, and $f\in A$ is a homogeneous element then $U=\Spec(A_f)$ is equivariant but usually not strongly equivariant. However, it is strongly equivariant if $f$ is of degree zero. In particular, an important particular case of Lemma~\ref{streqlem} is when $X_i$ are affine and each $X_{ij}$ is a localization of both $X_i$ and $X_j$ at elements of degree zero.
\end{example}

\subsubsection{Geometric quotients}
If the action of $G$ on $X$ is relatively affine and each fiber of the quotient morphism $p\:X\to X\sslash G$ consists of a single orbit then we say that the quotient is {\em geometric}\index{geometric quotient} and use the notation $X/G$ instead of $X\sslash G$. This matches the terminology, but not the notation, of GIT, see \cite[Definition 0.6]{GIT}.

\subsubsection{Special orbits}\label{Def:special-orbit}
If an orbit of a $G$-action on $X$ is closed in the fiber of $X\to X\sslash G$ then we say that the orbit is {\em special}.\index{special orbit} Obviously, such an orbit is a scheme.

\subsubsection{Local actions}\label{Sec:local-action}
We say that a relatively affine action of $G$ on a scheme $X$ is {\em local}\index{local action} if $X$ is quasi-compact and contains a single closed orbit.

\begin{lemma}\label{Lem:local-action}
Assume we have a relatively affine action of  $G=\bfD_L$  on a scheme $X$. Then the following conditions are equivalent:

(i) The action is local.

(ii) $X$ is affine, say $X=\Spec(A)$, and the $L$-graded ring $A$ is $L$-local.

(iii) The quotient $Y=X\sslash G$ is local.
\end{lemma}
\begin{proof}
A scheme is local if and only if it is quasi-compact and contains a single closed point. Therefore (i) and (iii) are equivalent. Equivalence of (ii) and (iii) was proved in Lemma~\ref{Lem:Llocal}.
\end{proof}

\subsubsection{Localization along a special orbit}\label{Sec:localspecial}
Assume that $O$ is a special orbit of a relatively affine action on $X$, and $y$ is its image in $Y$. Consider the localization $Y_y=\Spec(\cO_{Y,y})$ and set $X_O=X\times_YY_y$. We call $X_O$ the {\em equivariant localization} of $X$ along $O$. Note that $X_O\sslash G=Y_y$ by the universality of the quotient, in particular, $G$ acts locally on $X_O$.

\begin{remark}
(i) Set-theoretically, $X_O$ consists of all orbits whose closure contains $O$. So, even if $O=\{x\}$ is a closed point, it typically happens that $X_O$ is larger than the localization of $X$ at $x$. Equivariant localization of a scheme $X=\Spec(A)$ corresponds to homogeneous localization of $A$ in the sense of Example~\ref{localexam}.

(ii) Even if $O$ is only locally closed in the fiber of $X \to X\sslash G$, one can define an equivariant localization $X_O\into X$ whose only closed orbit is $O$. We will not use this construction. If $O$ is not special then the localization morphism $X_O\into X$ is not inert (see \S\ref{Sec:inert} below), and the morphism $X_O\sslash G\to Y$ can be bad (e.g. non-flat).
\end{remark}

\subsubsection{The schemes of fixed points}\label{fixedpointssec}
If $X$ is acted on by $G$ then the {\em scheme of fixed points} of $X$ is the maximal closed subscheme $X^G$ of $X$ such that $X^G$ is equivariant and the action of $G$ on it is trivial. In other words, it is the maximal closed subscheme over which the inclusion $I_X\into G\times X$ becomes an isomorphism. For diagonalizable groups, existence and functoriality of fixed points schemes is guaranteed by \cite[VIII.6.5(a)]{SGA3-2}, where one sets $Y = G$, $Z = X \times X$ and $Z'\subset Z$ the diagonal. If $X=\Spec(A)$ is affine then, as we noted in \S\ref{Sec:fixed}, $X^G=\Spec(A_G)$.

\subsubsection{Inertia stratification}
If $G$ is diagonalizable and $G'\subseteq G$ is a subgroup then $$X(G'):=X^{G'}\setminus\cup_{G''\supsetneq G'}X^{G''}$$ is the maximal $G$-equivariant subscheme $Y$ with constant inertia equal to $G'$ (i.e. such that $I_Y=G'\times Y$). The family of subschemes $\{X(G')\}_{G'\subseteq G}$ provides a $G$-equivariant stratification of $X$ that we call the {\em inertia stratification}.\index{inertia!stratification} Set-theoretically, this is the stratification of $X$ by the stabilizers of points.

\subsubsection{Regularity}
The fixed points functor preserves regularity, see \cite[Corollary of Theorem 5.4]{Fogarty}. We provide a simple proof in our situation.

\begin{proposition}\label{Prop:reg}
Assume that a diagonalizable group $G$ acts on a regular scheme $X$. Then the scheme of fixed points $X^G$ is regular. In particular, the strata of the inertia stratification of $X$ are regular.
\end{proposition}
\begin{proof}
The claim is local at a point $x\in X^G\into X$. Since $x$ is $G$-invariant, we can replace $X$ with the equivariant localization along $x$. Then $X$ is the spectrum of a strictly $L$-local ring, and it remains to use Lemma~\ref{coinvlem}.
\end{proof}

\subsubsection{The case of $\GGm$}\label{xplusminus}

We discuss a construction which is specific to $G=\GGm=\bfD_\ZZ$ and will be used in \cite{AT2}; we indicate the general case in Remark \ref{Rem:gen-GIT} below. Assume that $X$ is provided with a relatively affine action of $G$. Following \cite{W-Cobordism} we define $X_+$ (resp. $X_-$) to be the open subscheme obtained by removing all orbits that have a limit at $+\infty$ (resp. $-\infty$). The construction of $X_\pm$ is local on $X\sslash G$ and if the latter is affine, say $X=\Spec A$ and $X\sslash G=\Spec A_0$, then $X_+=X\setminus V(A_-)$ with $A_-=\oplus_{n<0}A_n$. Similarly, $X_-=X\setminus V(A_+)$ for $A_+=\oplus_{n>0}A_n$.

\begin{lemma}
Assume that $G=\GG_m$ acts in a  relatively affine manner on a scheme $X$. Then,

(i) $G$ acts in a relatively affine manner on $X_+$ and $X_-$.

(ii) If $X=\Spec(A)$ is affine then the schemes $X_f=\Spec(A[f^{-1}])$ with homogeneous $f\in A_-$ (resp. the schemes $X_f$ with $f$ homogeneous in $A_+$) form an open strongly equivariant covering of $X_+$ (resp. $X_-$).
\end{lemma}
\begin{proof}
If $f\in A_-$ then $V(A_-)\subseteq V(f)$ and hence $X_f$ is an equivariant open subscheme of $X_+$. By definition $X_+=\cup_{f\in A_m, m<0}X_f$. By Lemma~\ref{streqlem} it suffices to check that $X_{fg}$ is strongly equivariant in $X_f$ for $f\in A_n$ and $g\in A_m$ with $m,n<0$. It remains to notice that $X_{fg}$ can be described as the localization of $X_f$ obtained by inverting the degree-zero element $f^{-n}g^{m}\in A[f^{-1}]$, see Remark \ref{streqexam}.
\end{proof}

\begin{remark}[see  \cite{Thaddeus}]\label{Rem:gen-GIT}
For arbitrary $G = \bfD_L$ and affine $X=\Spec A$ one defines $X\sslash_mG = \Proj A[z]^G$ with $G$ acting on $z$ via $-m\in L$. Then there is a locally finite decomposition $\Sigma = \coprod \sigma_i^\circ$ of the monoid $\Sigma\subset L$ of characters figuring in $A$ into relative interiors of poyhedral cones, and  $X\sslash_mG= X\sslash_{\sigma^\circ}G$ is constant on such interiors of cones.  The replacement of $X_\pm$ is $X^{ss}(m)$, the semistable locus of $X$ with respect to the linearization provided by the character $m$.
\end{remark}

\subsection{Basic properties of the quotient functor}
In this section, we study the quotient map $X\to X\sslash G$ and properties of schemes and morphisms preserved by the quotient functor. The most subtle result is that the quotient map is finite when $X$ is noetherian.

\subsubsection{Quotients of schemes}

\begin{theorem}\label{quottheorem}
Assume that a diagonalizable group $G=\bfD_L$ acts trivially on a scheme $S$ and an $S$-scheme $X$ is provided with a relatively affine action of $G$, then:

(i) Assume that $X$ satisfies one of the following properties: (a) reduced, (b) integral, (c) normal with finitely many connected components, (d) locally of finite type over $S$, (e) of finite type over $S$, (f) quasi-compact over $S$, (g) locally noetherian, (h) noetherian. Then $X\sslash G$ satisfies the same property.

(ii) If $X$ is locally noetherian then the quotient morphism $X\to X\sslash G$ is of finite type.
\end{theorem}
\begin{proof}
Note that $X\to S$ factors through $Y=X\sslash G$ because the latter is the categorical quotient by Theorem~\ref{relafftheor}. Claim (f) is obvious since $X\to Y$ is onto. Furthermore, a morphism is of finite type if and only if it is quasi-compact and locally of finite type, hence (e) follows from (d) and (f). It remains to prove all assertions except (e) and (f). As shown in \cite[Section 2 of Chapter 0]{GIT}, claims (a), (b) and (c) hold for any categorical quotient, hence they are implied by Theorem~\ref{relafftheor}(i). Note that the assertion of (d) is local on $S$ hence we can assume that $S$ is affine. Furthermore, all assertions of the theorem we are dealing with are local on $Y$, so we can assume that $Y$ is affine. In this case, (d) was proven in Lemma~\ref{fgquotlem}, and (g), (h) and (ii) were proven in Proposition~\ref{noetherprop}.
\end{proof}

\subsubsection{Quotients of morphisms}\label{Sec:desreg}
We start with the following corollary of the above theorem.

\begin{corollary}\label{fgcor}
Assume that locally noetherian schemes $X$ and $X'$ are provided with relatively affine actions of a diagonalizable group $G$, and $f\:X'\to X$ is a $G$-equivariant morphism. If $f$ is of finite type then the quotient morphism $f\sslash G$ is of finite type.
\end{corollary}
\begin{proof}
The morphism $X\to X\sslash G$ is of finite type by Theorem~\ref{quottheorem}(ii). Hence the composition $X'\to X\sslash G$ is of finite type, and then $X'\sslash G\to X\sslash G$ is also of finite type by Theorem~\ref{quottheorem}(i)(d).
\end{proof}

\begin{proposition}\label{quotmorprop}
Let $G$ be a diagonalizable group and $f\:X'\to X$ a $G$-equivariant morphism such that the actions on $X$ and $X'$ are relatively affine. Then,

(i) If $f$ satisfies one of the following properties: (a) affine, (b) integral, (c) a closed embedding, then $f\sslash G$ satisfies the same property.

(ii) If $X$ is locally noetherian and $f$ is finite then $f\sslash G$ is finite.
\end{proposition}
\begin{proof}
The claim is local on $Y=X\sslash G$, hence we can assume that $X=\Spec(A)$ and $Y=\Spec(A_0)$ are affine. Since $f$ is affine in each case, $X'=\Spec(A')$ is affine, and so $Y'=\Spec(A'_0)$ is affine. This proves (a). If $A\to A'$ is onto then $A_0\to A'_0$ is onto and we obtain (c).

If $A\to A'$ is integral then any $x\in A'$ satisfies an integral equation $x^n+\sum_{i=0}^{n-1}a_ix^i$ with coefficients in $A$. If $x\in A'_0$ then replacing each $a_i$ with its component of degree zero we obtain an integral equation for $x$ with coefficients in $A_0$. Thus, $A'_0$ is integral over $A_0$ and we obtain (b). Finally, (ii) follows from (b) and Corollary~\ref{fgcor}.
\end{proof}

\subsubsection{Some bad examples}
Basic properties of $G$-morphisms, including, smoothness and separatedness, are not preserved by the quotient functor. Here are some classical bad examples.

\begin{example}\label{badexam}
(i) Let $k$ be a field with ${\rm char}\ k\neq 2$ and let $X=\Spec(k[x,y])$ be an affine plane with $\mu_2$ acting by changing the sign of $x$ and $y$. Then $Y=X\sslash\mu_2$ is the quotient singularity $Y=\Spec(k[x^2,xy,y^2])$ and the quotient morphism $X\to Y$ is not even flat. On the other hand, set $X'=X\times\mu_2$ with $\mu_2$ acting diagonally. The projection $X'\to X$ is a split \'etale covering which is not inert above the origin $O\in X$ because $O$ is fixed by $\mu_2$ and the action of $\mu_2$ on $X'$ is free. Moreover, $Y'=X'\sslash\mu_2\toisom X$, so the quotient map $f\sslash\mu_2$ is the projection $X\to Y$. Although $f$ is split \'etale, its quotient is not even flat.

(i)' Similarly, it is easy to construct a $\GGm$-action on $X=\bbA^3_k$ such that the quotient $Y=X\sslash\GGm$ is not smooth (e.g., if $X=\Spec(A)$ for $A=k[x,y,z]$ with $x,y\in A_1$ and $z\in A_{-2}$, then $Y=\Spec(k[x^2z,xyz,y^2z])$ is an orbifold). Then automatically the morphism $X\to Y$ is not flat, and setting $X'=X\times\GGm$ we obtain another bad example, where $X'\to X$ is smooth but the morphism between $X'\sslash\GGm=X$ and $X\sslash\GGm$ is not flat.

(ii) Let $k$ be a field and $G=\GG_m=\mathbf{D}_\ZZ$. Consider the $\ZZ$-graded $k$-algebra $A=k[x,y,z]$ with $x\in A_1$ and $y,z\in A_{-1}$, and set $X=\Spec(A)=\AA^3_k$. Then $Y=X\sslash G$ equals to $\Spec(k[xz,xy])=\AA^2_k$ and the orbits over the origin are: the origin, the punctured $(x)$-axis, all punctured lines through the origin in the $(yz)$-plane. Consider the equivariant subspace $X'=\Spec(k[x,y,z^{\pm 1}])$. The open embedding $X'\into X$ preserves stabilizers but takes some closed orbits to non-closed orbits. For example, the punctured $(z)$-axis is closed in $X'$. The quotient $Y'=X'\sslash G=\Spec(k[xz,\frac{y}{z}])$ is not flat over $Y$. In fact, it is an affine chart of the blowing up of $Y$ at the origin, and the quotient map $X'\to Y'$ separates all orbits of $X$ sitting over the origin of $Y$ and contained in $X'$.

(ii)' Analogous examples related to cobordisms will play a crucial role in the proof of the factorization theorem: $(B_a)_+\into B_a$ is an open embedding preserving the stabilizers (see \cite[\S3.4]{AT2}), but $(B_a)_+\sslash G\to B_a\sslash G$ is usually a non-trivial modification.

(iii) Let $G=\GGm$ act on $T=\Spec(k[x,y])$ as $x\mapsto tx$ and $y\mapsto t^{-1}y$, and let $X'$ be obtained from $T$ by removing the origin. Then $X'\sslash G$ is an affine line with doubled origin. In particular, the morphism $f\:X'\to\Spec(k)$ is separated but its quotient is not.
\end{example}

\subsection{Strongly equivariant morphisms}
The situation improves drastically if one considers quotients of a more restrictive class of strongly equivariant morphisms.

\subsubsection{The definition}\label{strongsec}
We say that a $G$-morphism $f\:X'\to X$ is {\em strongly equivariant}\index{strongly equivariant morphism} if the actions are relatively affine and $f$ is the base change of its quotient $f\sslash G$, i.e. the morphism $\phi\:X'\to X\times_YY'$ is an isomorphism, where $Y=X\sslash G$ and $Y'=X'\sslash G$. Furthermore, we say that $f$ is {\em strongly equivariant over} a point $y'\in Y'$ if $\phi$ is an isomorphism over $y'$ in the sense that $$\phi_{y'}:X'\times_{Y'}\Spec(\cO_{Y',y'})\to X\times_{Y}\Spec(\cO_{Y',y'})$$ is an isomorphism. Note that $f$ is strongly equivariant if and only if it is strongly equivariant over all points of $Y'$, and if $Y'$ is quasi-compact then it suffices to consider only closed points.

\begin{lemma}\label{stronglem}
Let $G=\bfD_L$ be a diagonalizable group.

(i) The composition of strongly $G$-equivariant morphisms is strongly $G$-equivariant.

(ii) If $Y\to X$ is a strongly $G$-equivariant morphism and $g\:Z\to Y$ is a $G$-equivariant morphism such that the composition is strongly $G$-equivariant, then $g$ is strongly $G$-equivariant.

(iii) If $Y\to X$ is strongly $G$-equivariant and $Z\to X$ is $G$-equivariant then the base change $Y\times_XZ\to Z$ is strongly $G$-equivariant.

(iv) If $f\:Y\to X$ is strongly equivariant then the diagonal $\Delta_f\:Y\to Y\times_XY$ is strongly equivariant and $\Delta_f\sslash G$ is the diagonal of $f\sslash G$.
\end{lemma}
\begin{proof}
The proof of (i), (ii) and (iv) is a simple diagram chase of  the relevant cartesian squares, so we omit it. Let us prove (iii). The claim is local on $X\sslash G$, $Y\sslash G$ and $Z\sslash G$ hence we can assume that $X=\Spec(A)$, $Y=\Spec(B)$ and $Z=\Spec(A')$ for $L$-graded rings $A,B,A'$. We are given that $B=A\otimes_{A_0}B_0$ and we should prove that $B'=A'\otimes_AB$ satisfies $B'=A'\otimes_{A'_0}B'_0$. Clearly, $B'=A'\otimes_{A_0}B_0$ and so $B'_n=A'_n\otimes_{A_0}B_0$ for any $n\in L$. Therefore, $B'_n=A'_n\otimes_{A'_0}A'_0\otimes_{A_0}B_0=A'_n\otimes_{A'_0}B'_0,$ as required.
\end{proof}

\subsubsection{Strongly satisfied properties}
Let $P$ be a property of morphisms preserved by base changes. As in Section \ref{strhomsec}, we say that an equivariant morphism $f\:X'\to X$ {\em strongly} satisfies $P$ if it is strongly equivariant and the quotient $f\sslash G$ satisfies $P$. In particular, $f$ itself satisfies $P$.

\begin{remark}
(i) In the case of strongly \'etale\index{strongly etale morphism@strongly \'etale morphism} morphisms we recover the definition from \cite[Appendix 1.D]{GIT}. Such morphisms played important role in \cite{AKMW}, and we will use strongly regular\index{strongly regular morphism} morphisms in \cite{AT2} for analogous purposes.

(ii) Let $U$ be an open subscheme of $X$. The morphism $U\into X$ is a strongly open immersion if and only if $U$ is a strongly equivariant open subscheme.

(iii) For various properties $P$ it is true that a morphism strongly satisfies $P$ if and only if it satisfies $P$ and is strongly equivariant. In fact, one should only check that if $f$ satisfies $P$ and is strongly equivariant then $f\sslash G$ satisfies $P$. However, such descent results may be difficult to prove because the morphisms $X\to X\sslash G$ are not always flat; in particular, flat descent is unapplicable. For the strong \'etale property such descent claim is a part of Luna's fundamental lemma.
\end{remark}

\subsubsection{Some descent results}
Here is a (rather incomplete) list of properties for which the descent is easy.

\begin{proposition}\label{sepandprop}
Let $G$ be a diagonalizable group and $f\:X'\to X$ a $G$-equivariant morphism of schemes with relatively affine $G$-actions. Then,

(i) Let $P$ be any of the following properties: (a) has finite fibers, (b) a monomorphism, (c) separated, (d) universally closed. Then $f$ strongly satisfies $P$ if and only if it satisfies $P$ and is strongly equivariant.

(iii) Assume that $X$ is locally noetherian and let $P$ be one of the following properties: (e) quasi-finite, (f) proper. Then $f$ strongly satisfies $P$ if and only if it satisfies $P$ and is strongly equivariant.
\end{proposition}
\begin{proof}
In all claims we assume that $f$ is strongly equivariant and satisfies $P$ and we should prove that $g=f\sslash G$ satisfies $P$. The assertion of (a) follows from the surjectivity of $X\to X\sslash G$. Note that a morphism is a monomorphism if and only if its diagonal is an isomorphism. Hence (b) follows from Lemma~\ref{stronglem}(iv). Similarly, (c) follows from Lemma~\ref{stronglem}(iv) and Proposition~\ref{quotmorprop}(i)(c).

(d) First, assume only that $f$ is closed. Then for any closed set $T\subset X'\sslash G$ the preimage of $g(T)$ in $X$ is closed. Since the quotient morphism $X\to X\sslash G$ is submersive, $g(T)$ is closed. Thus $g$ is closed. The assertion of (ii) follows from this and the fact that the base change of $f$ by a strongly equivariant morphism is strongly equivariant by Lemma~\ref{stronglem}(iii).

Finally, (e) follows from (a) and Corollary~\ref{fgcor}, and (f) follows from (d) and Corollary~\ref{fgcor}.
\end{proof}

\subsubsection{The case of $\GGm$}
Assume that $X$ and $Y$ are provided with a relatively affine action of $G=\GGm$. We refer to Section~\ref{xplusminus} for the definitions of $X_\pm$ and $Y_\pm$.

\begin{lemma}\label{xpluslem}
Keeping the above notation, if $f\:X\to Y$ is strongly equivariant then $X_\pm=Y_\pm\times_YX$.
\end{lemma}
\begin{proof}
The claim is local on $Y\sslash G$ and $X\sslash G$, so we can assume that $Y=\Spec B$, $Y\sslash G=\Spec B_0$, $X=\Spec A$ and $X\sslash G=\Spec A_0$. By strong equivariance, $A=B\otimes_{B_0}A_0$. Hence $A_n=B_n\otimes_{B_0}A_0$ for any $n$ and we obtain equalities of ideals $B_-A=A_-$ and $B_+A=A_+$.
\end{proof}
\begin{remark} For a general $G=\bfD_L$, we similarly have that $$X^{ss}(m)=(Y^{ss}(m))\times_YX.$$
\end{remark}
\subsection{Free actions}
In \S\ref{Sec:local} we studied the quotients when the stabilizers are maximal. This section is devoted to the other extreme case when the stabilizers are trivial. We will show in Corollary~\ref{Cor:freeact} that any quotient of a relatively affine action of $G = \bfD_L$ can be described in terms of these two cases.

\subsubsection{Definitions}
We say that an action of $G$ on $X$ is {\em regular} or {\em split free}\index{split free action} if there is an equivariant isomorphism $X\toisom G\times Y$, where $Y=X\sslash G$. If such an isomorphism only exists locally on $Y$ for a topology $\tau$ (e.g. flat, \'etale, or Zariski) then we say that the action is $\tau$-split free. Finally, if there is a $\tau$-open morphism $g\:Y'\to Y$ and an equivariant isomorphism $X\times_YY'\toisom G\times Y'$ then for any point $y\in g(Y')$ we say that the action is {\em $\tau$-split free over} $y$.

\begin{remark}
Recall that an action of $G$ on $X$ is {\em semi-regular} or {\em free} if the morphism $G\times X\to X\times X$ is a closed embedding, and this condition is equivalent to the condition that $\psi\:G\times X\to X\times_YX$ is an isomorphism; in other words, $X$ is a pseudo $G$-torsor over $Y$, see \cite[{\tt tag/0498}]{stacks}. This pseudo-torsor is a $\tau$-torsor of $G$ if and only if the action is $\tau$-split free. In fact, any semi-regular action has free orbits hence it follows from Lemma~\ref{Lem:freeact} below that it is flat-split free over the quotient $Y$. In particular, we will not distinguish free and flat-split free actions.
\end{remark}

\subsubsection{A criterion for splitting}
Flat-split freeness of an action can be tested very easily: the stabilizers of the fibers should be trivial.

\begin{lemma}\label{Lem:freeact}
Given a relatively affine action of a diagonalizable group $G=\mathbf{D}_L$ on a scheme $X$, let $f\:X\to Y=X\sslash G$ be the quotient morphism and $y\in Y$  a point. Then the following conditions are equivalent:
\begin{enumerate}
\item[(i)]
 the action is flat-split free over $y$,
\item[(ii)]
 the scheme-theoretic fiber $X_y=y\times_YX$ coincides with a single free orbit $O=\Spec(k[L])$,
\item[(iii)]
 the set-theoretic fiber $f^{-1}(y)$ is a single free orbit,
\item[(iv)]
 all points of the fiber $f^{-1}(y)$ have trivial stabilizer.
\end{enumerate}
Moreover, if the degree of the torsion of $L$ is invertible on $Y$ (respectively, $L$ is torsion free) then one may replace the flat topology in (i) with the \'etale (respectively, Zariski) topology.
\end{lemma}
\begin{proof}
The implications (i)$\implies$(ii)$\implies$(iii)$\implies$(iv) are obvious. In the opposite direction, if (iv) holds then all orbits in $f^{-1}(y)$ are free and closed, hence the fiber is a single orbit by Theorem~\ref{relafftheor}(ii), and we obtain (iii).

If (iii) holds then $X_y=\Spec(A)$ such that $A_0=k$ and $k[L]$ is the reduction of $A$. This implies that each $A_n$ contains a unit, and hence is an invertible $k$-module. Therefore, the reduction $A\to k[L]$ is an isomorphism, and we obtain (ii).

Finally, assume that $X_y=\Spec(k[L])$. Let $L=\oplus_{i=1}^nL_i$ be a decomposition of $L$ into cyclic groups, and choose a generator $m_i$ of $L_i$. Shrinking $Y$ around $y$ we can assume that $X=\Spec(A)$, $Y=\Spec(A_0)$, and each $m_i\in k[L]$ lifts to a unit $u_i\in A_{m_i}$ with respect to the homomorphism $A\to k[L]$. For each $m_i$ of a finite order $d_i$ we have that $u_i^{d_i}=a_i\in A_0$. Obviously, $X$ is a principal $G$-torsor over $Y$ which is trivialized by adjoining the roots $a_i^{1/d_i}$ to $A_0$. The latter defines a flat covering of $Y$, which is \'etale (respectively, an isomorphism) if $d_i$ are invertible on $Y$ (respectively, $L$ is torsion free).
\end{proof}

The following corollary shows that locally one can construct quotients in two steps: first by dividing by the stabilizer and then by dividing by a locally free action.

\begin{corollary}\label{Cor:freeact}
Assume that a scheme $X$ is provided with a relatively affine action of a diagonalizable group $G=\mathbf{D}_L$, $y$ is a point of $Y=X\sslash G$, $G'$ is the stabilizer of the closed orbit of the fiber $X_y$, $Z=X\sslash G'$, and $G'' = G/G'$. Then $Z\sslash G''=Y$ and the $G''$-action on $Z$ is flat-split free over $y$. Moreover, if the torsion of $L$ is invertible on $Y$ (resp. $L$ is torsion free) then the action is \'etale-split (resp. Zariski-split) free over $y$.
\end{corollary}
\begin{proof}
By universality, quotients are compatible with taking fibers. So, in view of Lemma~\ref{Lem:freeact} it suffices to show that $X_y\sslash G'$ is a single free $G''$-orbit. Note that $X_y=\Spec(A)$, where $(A,m)$ is an $L$-local ring with $A_0=k$. Let $0\to L''\to L\to L'\to 0$ be the exact sequence corresponding to $1\to G'\to G\to G''\to 1$. Since $A/m=k[L'']$, each $A_n$ with $n\in L''$ contains a unit and we obtain that $A^{L'}=\oplus_{n\in L''}k=k[L'']$. Thus,
$X_y\sslash G'=\Spec(k[L''])$, and we are done.
\end{proof}

\subsubsection{Descent of regularity}
We show  below that dividing by a free action is a nice functor that preserves various properties of equivariant morphisms. In fact, the following lemma reduces this to the usual flat descent.

\begin{lemma}\label{Lem:freequotflat}
Assume that a diagonalizable group $G$ acts on a scheme $X$ and the action is flat-split free over a point $y\in Y=X\sslash G$. Then the quotient morphism $f\:X\to Y$ is flat along the fiber $f^{-1}(y)$. In particular, if $X$ is regular at a point $x\in f^{-1}(y)$ then $Y$ is regular at $y$.
\end{lemma}
\begin{proof}
By definition, there exists a flat morphism $Y'\to Y$, whose image contains $y$, and an isomorphism $X\times_YY'\toisom G\times Y'$. Then the flat base change $X\times_YY'\to Y'$ of $f$ is flat, and hence $f$ is flat over $y$. If $X$ is regular at a point $x$ then $Y$ is regular at $y$ by \cite[{\tt tag/00OF}]{stacks}.
\end{proof}

\begin{corollary}\label{Cor:freequotcor}
Assume that schemes $X$ and $X'$ are provided with relatively affine actions of a diagonalizable group $G=\mathbf{D}_L$, and $f\:X'\to X$ is a $G$-equivariant morphism with quotient $g=f\sslash G\:Y'\to Y$. Suppose $y'\in Y'$ is a point such that the actions on $X'$ and $X$ are flat-split free over $y'$ and $g(y')$, respectively. Then $f$ is strongly equivariant over $y'$. In particular, if $f$ is regular or lci at a point $x'\in X'_{y'}$ then $g$ is regular or lci at $y'$, respectively.
\end{corollary}
\begin{proof}
The statement is local at $y'$ and $y$. So, we can assume that $Y$ and $Y'$ are local, and the claim reduces to showing that $\phi\:X'\to X\times_YY'$ is an isomorphism. By flat descent, this can be checked flat-locally on $Y$, so by Lemma~\ref{Lem:freequotflat} we can assume that $X\toisom G\times Y$. By the same lemma, there exists a flat covering $Y''\to Y'$ such that $X'\times_{Y'}Y''\toisom G\times Y''$. The base change of $\phi$ with respect to the morphism $Y''\to Y'$ is the isomorphism $X'\times_{Y'}Y''\toisom G\times Y''\toisom X\times_YY''$. Thus, $\phi$ is an isomorphism by flat descent.

Suppose $f$ is regular. To check that $g: Y' \to Y$ is regular or lci it suffices to check that its pullback by a flat surjective morphism is regular or lci, respectively; but  $X \to Y$ is flat and surjective by Lemma~\ref{Lem:freequotflat}, and we have shown that the pullback is $f:X'\to X$.
\end{proof}

\subsection{Inert morphisms}\label{Sec:inert}
The main disadvantage in the definition of strong equivariance (Section \ref{strongsec}) is that it involves the quotient morphism. It is desirable to have an explicit criterion of strong equivariance in terms of the morphism itself. This leads to the notion of inert morphisms.

\subsubsection{Inertia preserving morphisms}\label{Sec:inertia-preserving}
Recall that a $G$-equivariant morphism $f\:X'\to X$ is called {\em fixed point reflecting} if for any $x'\in X'$ and $x=f(x')$ the inclusion of the stabilizers $G_{x'}\into G_x$ is an equality (e.g. \cite[IV.1.8]{Knutson}). We propose the following scheme-theoretic strengthening of this notion: a $G$-equivariant morphism $f\:X'\to X$ is {\em inertia preserving} if the closed immersion $I_{X'}\into I_X\times_XX'$ is an equality.

\begin{lemma}\label{inertialem}
Assume that locally noetherian schemes $X$ and $X'$ are provided with relatively affine actions of a diagonalizable group $G=\mathbf{D}_L$, and $f\:X'\to X$ is a $G$-equivariant morphism. Consider the following conditions:

(i) $f$ is inertia preserving.

(ii) For each subgroup $H\subseteq G$ the closed immersion $X'^H\into X^H\times_XX'$ is an equality.

(iii) For each point $x\in X$ the stabilizer $G_x$ acts trivially on the fiber $X'\times_X x$.

(iv) $f$ is fixed point reflecting and for any point $x'\in X'$ with $x=f(x')$ the action of $G_x$ on $\Omega_{X'/X}\otimes k(x')$ is trivial.

(v) $f$  is fixed point reflecting.

Then Conditions (ii), (iii), (iv) are equivalent, each condition is implied by (i), and all conditions imply (v). In addition, if $f$ has reduced fibers then (v) implies the conditions (ii)--(iv).
\end{lemma}
\begin{proof}
The implications (i)$\implies$(ii)$\implies$(iii) and (iv)$\implies$(v) are obvious, and the equivalence (iii)$\Longleftrightarrow$(iv) follows from Lemma~\ref{fiblem} since $\Omega_{X'/X}\otimes k(x')=\Omega_{Z_x/x}\otimes k(x')$ for the fiber $Z_x=X'\times_X x$.

We show (iii)$\Longrightarrow$(ii): Assume that (ii) fails, i.e. there exists $H=\bfD_{L'}\subseteq G$ which acts non-trivially on $X^H\times_XX'$. It suffices to prove that in this case $H$ acts non-trivially already on the fiber $Z_x$ over a point $x\in X^H$. Choose an open affine subscheme $\Spec(A)\into X^H$ such that the action of $H$ on $\Spec(B)=\Spec(A)\times_XX'$ is non-trivial. The morphism $\Spec(B)\to\Spec(A)$ corresponds to an $L'$-homogeneous homomorphism $A\to B$ such that $A=A_0$ and $B_l\neq 0$ for some $l\neq 0$. Note that $B_l$ is finitely generated over $B_0$ since $B$ is noetherian, hence there exists a point $x\in\Spec(A)$ such that $B_l\otimes_A k(x)\neq 0$. Then the action on the fiber over $x$ is non-trivial.

Finally, assume that the fibers of $f$ are reduced, and let us show that in this case (v) implies (iii). Both conditions are fiberwise, so replacing $f$ and $G$ with the fiber over a point $x$ and the stabilizer $G_s$, we can assume that $X=\Spec(k)$ is $G$-invariant and for any point $y\in Y=\Spec(A)$ one has that $G_y=G$. Fix $0\neq n\in L$, then the latter condition implies that $A_n$ is contained in all $L$-prime ideals of $A$. By Lemma~\ref{homradlem} $A_n$ lies in the nilpotent radical, and since $A$ is reduced by the assumption, $A=A_0$, as claimed.
\end{proof}

\subsubsection{Inert morphisms}\label{Sec:inert-def}
Consider $G$-equivariant morphism $f\:X'\to X$ and {\em assume that it takes special orbits (\ref{Def:special-orbit}) to special orbits.} We say that
\begin{enumerate}
\item the morphism $f$ is {\em pointwise inert}\index{pointwise inert morphism}  if condition (v)  of Lemma~\ref{inertialem} holds, namely it is fixed-point reflecting;
\item the morphism is {\em fiberwise inert}\index{fiberwise inert} if either condition (ii), (iii) or (iv) of Lemma~\ref{inertialem} holds, namely, for each point $x\in X$ the stabilizer $G_x$ acts trivially on the fiber $X'\times_X x$;
\item the morphism is {\em inert}\index{inert morphism}  if condition (i)  of Lemma~\ref{inertialem} holds, namely it is inertia preserving.
\end{enumerate}

\begin{remark}\label{inertrem2}
(i) One can easily construct a morphism (with non-reduced fibers) which is pointwise inert but not fiberwise inert. We have not studied the difference (if any!) between inert morphisms and morphisms which are fiberwise inert.

(i) It follows directly from the definition (\ref{strongsec}) that any strongly equivariant morphism is inert and, conversely, pointwise inertness will play a role in our criteria of strong equivariance.

(ii) If either of the conditions of inertness is violated then the quotient functor might behave rather badly. For example, the morphism is not inertia preserving in Example~\ref{badexam}(i) and (i)', and it takes special orbits to non-special orbits in Example~\ref{badexam}(ii) and (ii)'.
\end{remark}

\subsubsection{Comparison with the literature}
In the case of a reductive group $G$ and a $G$-equivariant \'etale morphism of varieties, the fixed point reflecting condition was used by Luna. Although not formulated separately, it is contained in \cite[Lemme fondamental]{Luna}. For finite groups, fixed point reflecting morphisms were called inert in \cite[Section~VIII.5.3.6]{Illusie-Temkin}. In both cases, the condition on special orbits was not required since it was automatically satisfied. An analogue of our version of inert morphism was called ``schematically inert" in \cite[Remark~VIII.5.6.2]{Illusie-Temkin}.

An analogue of inertness and pointwise inertness is used by Alper in a much more general context of good and adequate moduli spaces, see for example \cite{Alper}. It is needed to guarantee that the moduli space functor (which generalizes the quotient functor) behaves nicely.

\subsection{Main results about the quotient functor}
In this section we will prove that pointwise inert regular morphisms are strongly regular. In the case of diagonalizable groups, this extends the classical Luna's fundamental lemma, which deals with inert \'etale morphisms of varieties, to regular morphisms of schemes. Moreover, we will obtain general criteria for equivariant morphisms to be strongly equivariant and establish descent for syntomic and lci morphisms.

\subsubsection{Local case}
We start with the following extension of Theorem \ref{localcaseth} to the case of local actions that do not have to be strictly local.

\begin{lemma}\label{localLuna}
Assume that locally noetherian schemes $X$ and $X'$ are provided with relatively affine actions of a diagonalizable group $G=\mathbf{D}_L$, and $f\:X'\to X$ is a $G$-equivariant morphism with quotient $g=f\sslash G\:Y'\to Y$. Let $x'\in X'$ be a point with images $x=f(x')\in X$, $y'\in Y'$, and $y\in Y$. Assume that
\begin{itemize}
\item[(a)] $x$ lies in a special orbit,
\item[(b)] the inclusion of the stabilizers $G_{x'}\subseteq G_x$ is an equality,
\item[(c)] let $G_x=\mathbf{D}_{L_x}$, then the $L_x$-grading on the vector spaces $\Omega_{X'/X}\otimes k(x')$ and $H_1(\LL_{X'/X}\otimes^{\rm L} k(x'))$ are trivial.
\end{itemize}
Then
\begin{itemize}
\item[(i)] $f$ is strongly equivariant over $y'$,
\item[(ii)] $f$ is formally smooth at $x'$ if and only if $g$ is formally smooth at $y'$,
\item[(iii)] If $f$ is lci at $x'$ then $g$ is lci at $y'$. Conversely, if $g$ is lci at $y'$ and the homomorphisms $\cO_{Y,y}\to\cO_{Y',y'}$ and $\cO_{Y,y}\to\cO_{X,x}$ are $\Tor$-independent then $f$ is lci at $x'$. In particular, $f$ is syntomic at $x'$ if and only if $g$ is syntomic at $y'$.
\end{itemize}
\end{lemma}
\begin{proof}
The claim is local at $y'$ and $y$, so we can assume that $Y'=\Spec\cO_{Y',y'}$ and $Y=\Spec\cO_{Y,y}$. Then the actions on $X'$ and $X$ are local (\ref{Sec:local-action}), and $x$ lies in the closed orbit by assumption (a). We need to prove that $f$ is strongly equivariant, and if $f$ is formally smooth then $g$ is formally smooth. Let us recall two cases of the lemma that were already established earlier.

Case 1. {\em The lemma holds true when $G_x=G$.} Indeed, in this case the closed orbits are $x$ and $x'$, so $X=\Spec(A)$ and $X'=\Spec(A')$, where $A$ and $A'$ are strictly $L$-local rings by assumption (b). So Case 1 is covered by Theorem~\ref{localcaseth}.

Case 2. {\em The lemma holds true when $G_x=\{1\}$.} Indeed, in this case the actions are free by Lemma~\ref{Lem:freeact}, hence Case 2 is covered by Corollary~\ref{Cor:freequotcor}.

Assume now that $G_x$ is arbitrary. Recall that the actions of $G':=G/G_x$ on $Z=X\sslash G_x$ and $Z'=X'\sslash G_x$ are free by Corollary~\ref{Cor:freeact}. Consider the $G'$-equivariant morphism $h=f\sslash G_x$. By Case 1 above, $f$ is strongly equivariant over $z'$ and if $f$ is formally smooth at $x'$ then $h$ is formally smooth at its image $z'\in Z'$. By the local freeness of the action, $Z_y$ consists of a single orbit, so  $G'$, $h$, and $z'$ satisfy all assumptions of the lemma. Therefore, Case 2 applies and we obtain that $h$ is strongly equivariant over $y'$ and if $f$ is formally smooth at $x'$ then $g=h\sslash G'$ is formally smooth at $y'$. The lemma follows.
\end{proof}

\subsubsection{Luna's fundamental lemma}

\begin{theorem}\label{Th:Luna}
Let $X$ and $X'$ be locally noetherian schemes provided with relatively affine actions of a diagonalizable group $G$, and let $f\:X'\to X$ be a $G$-equivariant morphism. Then,

(i) $f$ is strongly equivariant if and only if $f$ is fiberwise inert and any special orbit in $X'$ contains a point $x'$ such that the action of $G_{x'}$ on $H_1(\LL_{X'/X}\otimes^{\rm L} k(x'))$ is trivial.

(ii) Let $P$ be one of the following properties: (a) regular, (b) smooth, (c) \'etale, (d) an open embedding. Then the following conditions are equivalent: (1) $f$ is pointwise inert and satisfies $P$,  (2) $f$ is inert and satisfies $P$, (3) $f$ is strongly equivariant and satisfies $P$, (4) $f$ strongly satisfies $P$.

(iii) $f$ is strongly syntomic if and only if it is syntomic and strongly equivariant. Moreover, if $f$ is strongly equivariant then the following claims hold: (a) if $f$ is lci then $f\sslash G$ is lci, (b) if $f\sslash G$ is lci and $\Tor$-independent with the morphism $X \to X\sslash G$ then $f$ is lci.
\end{theorem}

Part (iii) refers to Avramov's definition of lci morphisms and the resulting notion of syntomic morphisms described in Section \ref{Sec:lci-syntomic}.

\begin{proof}
(i) The direct implication is clear, so let us prove the inverse implication.  Let $y'\in X'\sslash G$ be a point and choose a point $x'$ in the special orbit over $y'$ such that the action of $G_{x'}$ on $H_1(\LL_{X'/X}\otimes^{\rm L} k(x'))$ is trivial. We assume $f$ is inert, so (a) $x=f(x')$ lies in a special orbit, (b) $G_{x'} = G_{x}$ and (c)  by Lemma~\ref{inertialem} the action of $G_{x'}$ on $H_0(\LL_{X'/X}\otimes^{\rm L} k(x'))=\Omega_{X'/X} \otimes k(x')$ is trivial as well. Thus we may apply Lemma \ref{localLuna} at $x'$, hence $f$ is strongly equivariant over $y'$, thus proving (i). Moreover, this argument shows that if $f$ is regular then $g=f\sslash G$ is formally smooth at all point of $X\sslash G$ and hence is itself a regular morphism.

Part (iii) follows similarly from Lemma \ref{localLuna}.

(ii) The implications (4)$\implies$(3)$\implies$(2)$\implies$(1) are obvious, so let us prove that (1)$\implies$(4). Regardless to the choice of $P$, the morphism $f$ is regular. Since $f$ has reduced fibers, it is fiberwise inert by Lemma~\ref{inertialem}, and as we have proved above it is then strongly regular. This covers (a). A morphism is smooth if and only if it is regular and of finite type. Hence (b) follows from (a) and Corollary~\ref{fgcor}. A smooth morphism is \'etale or an open embedding if and only if it has finite fibers or is a monomorphism, respectively. Hence (c) and (d) follow from (a) and assertions (a) and (b) of Proposition~\ref{sepandprop}.
\end{proof}

\subsection{Groups of multiplicative type}\label{nonsplitsec}
Assume that $S$ is a scheme and $G_S$ is an $S$-group. We say that $G_S$ is of {\em multiplicative type} if it is of finite type and $G_{S'}=G_S\times_SS'$ is isomorphic to a diagonalizable $S'$-group $\bfD_L\times S'$ for a surjective flat finitely presented morphism $h\:S'\to S$. For example, this includes the case when $S=\Spec(\ZZ[\frac{1}{N}])$ and $G_S=\ZZ/N\ZZ$. Our definition is more restrictive than in \cite[IX.1.1]{SGA3-2} because we use the fppf topology instead of the fpqc topology and consider only groups of finite type. Recall that by \cite[X.4.5]{SGA3-2} any such group $G_S$ is quasi-isotrivial, i.e. it can be split by an \'etale surjective morphism $h\:S'\to S$.

We claim that all results of Section~\ref{Lunasection} extend to the case when an $S$-scheme $X$ is provided with a relatively affine action of an $S$-group $G_S$ of multiplicative type. The only exception is part (ii) of Lemma~\ref{Lem:local-action}, which makes no sense for non-diagonalizable group. The proofs we provided of equivalence of (i) and (iii) in Lemma~\ref{Lem:local-action} and parts (a), (b) and (c) of Theorem~\ref{quottheorem}(i) apply to groups of multiplicative type without change. All remaining results, including Luna's fundamental lemma, can be extended to $G_S$ of multiplicative type with $L$ denoting the abelian group Cartier dual to $G_{S'}$ by use of \'etale descent with respect to $h$. Namely, $X_{S'}\sslash G_{S'}=(X\sslash G_S)\times_SS'$ so the result for $X_{S'}\sslash G_{S'}$ extends to $X\sslash G$ by descent. In fact, already flat descent suffices in almost all cases, with the main exception being Proposition~\ref{Prop:reg}.

\bibliographystyle{amsalpha}
\bibliography{factor-qe}

\providecommand{\bysame}{\leavevmode\hbox to3em{\hrulefill}\thinspace}
\providecommand{\MR}{\relax\ifhmode\unskip\space\fi MR }
\providecommand{\MRhref}[2]{%
  \href{http://www.ams.org/mathscinet-getitem?mr=#1}{#2}
}
\providecommand{\href}[2]{#2}
\begin{thebibliography}{AKMW02}

\bibitem[AFH94]{AFH}
Luchezar~L. Avramov, Hans-Bj{\o}rn Foxby, and Bernd Herzog, \emph{Structure of
  local homomorphisms}, J. Algebra \textbf{164} (1994), no.~1, 124--145.
  \MR{1268330 (95f:13029)}

\bibitem[AKMW02]{AKMW}
Dan Abramovich, Kalle Karu, Kenji Matsuki, and Jaros{\l}aw W{\l}odarczyk,
  \emph{Torification and factorization of birational maps}, J. Amer. Math. Soc.
  \textbf{15} (2002), no.~3, 531--572 (electronic). \MR{1896232 (2003c:14016)}

\bibitem[Alp08]{Alper}
Jarod Alper, \emph{Good moduli spaces for {A}rtin stacks}, ProQuest LLC, Ann
  Arbor, MI, 2008, Thesis (Ph.D.)--Stanford University. \MR{2711743}

\bibitem[Alp10]{Alper-local-structure}
\bysame, \emph{On the local quotient structure of {A}rtin stacks}, J. Pure
  Appl. Algebra \textbf{214} (2010), no.~9, 1576--1591. \MR{2593684
  (2011i:14021)}

\bibitem[And74]{Andre}
Michel Andr{\'e}, \emph{Localisation de la lissit\'e formelle}, Manuscripta
  Math. \textbf{13} (1974), 297--307. \MR{0357403 (50 \#9871)}

\bibitem[AOV08]{AOV1}
Dan Abramovich, Martin Olsson, and Angelo Vistoli, \emph{Tame stacks in
  positive characteristic}, Ann. Inst. Fourier (Grenoble) \textbf{58} (2008),
  no.~4, 1057--1091. \MR{2427954 (2009c:14002)}

\bibitem[AT15a]{AT2}
Dan Abramovich and Michael Temkin, \emph{Functorial factorization of birational
  maps for qe schemes in characteristic 0}, manuscript in preparation, 2015.

\bibitem[AT15b]{AT1}
\bysame, \emph{Torification of diagonalizable group actions on toroidal
  schemes}, manuscript in preparation, 2015.

\bibitem[Avr99]{Avramov-lci}
Luchezar~L. Avramov, \emph{Locally complete intersection homomorphisms and a
  conjecture of {Q}uillen on the vanishing of cotangent homology}, Ann. of
  Math. (2) \textbf{150} (1999), no.~2, 455--487. \MR{1726700 (2001a:13024)}

\bibitem[BR85]{Bardsley-Richardson}
Peter Bardsley and R.~W. Richardson, \emph{\'{E}tale slices for algebraic
  transformation groups in characteristic {$p$}}, Proc. London Math. Soc. (3)
  \textbf{51} (1985), no.~2, 295--317. \MR{794118 (86m:14034)}

\bibitem[Fog73]{Fogarty}
John Fogarty, \emph{Fixed point schemes}, Amer. J. Math. \textbf{95} (1973),
  35--51. \MR{0332805 (48 \#11130)}

\bibitem[FR93]{Franco-Rodicio}
L.~Franco and A.~G. Rodicio, \emph{On the vanishing of the second
  {A}ndr\'e-{Q}uillen homology of a local homomorphism}, J. Algebra
  \textbf{155} (1993), no.~1, 137--141. \MR{1206626 (94c:13011)}

\bibitem[Gro67]{ega}
A.~Grothendieck, \emph{\'{E}l\'ements de g\'eom\'etrie alg\'ebrique.}, Inst.
  Hautes \'Etudes Sci. Publ. Math. (1960-1967).

\bibitem[GY83]{Goto-Yamagishi}
Shiro Goto and Kikumichi Yamagishi, \emph{Finite generation of {N}oetherian
  graded rings}, Proc. Amer. Math. Soc. \textbf{89} (1983), no.~1, 41--44.
  \MR{706507 (84j:13015)}

\bibitem[Har77]{Hartshorne}
Robin Hartshorne, \emph{Algebraic geometry}, Springer-Verlag, New York, 1977,
  Graduate Texts in Mathematics, No. 52. \MR{0463157 (57 \#3116)}

\bibitem[Ill71]{Illusie-cotangent}
Luc Illusie, \emph{Complexe cotangent et d\'eformations. {I}}, Lecture Notes in
  Mathematics, vol. 239, Springer-Verlag, Berlin, 1971.

\bibitem[ILO12]{Illusie-Temkin}
L.~{Illusie}, Y.~{Laszlo}, and F.~{Orgogozo}, \emph{{Travaux de Gabber sur
  l'uniformisation locale et la cohomologie etale des schemas quasi-excellents.
  Seminaire a l'Ecole polytechnique 2006--2008}}, July 2012,
  \url{http://arxiv.org/abs/1207.3648}.

\bibitem[Knu71]{Knutson}
Donald Knutson, \emph{Algebraic spaces}, Lecture Notes in Mathematics, Vol.
  203, Springer-Verlag, Berlin, 1971. \MR{0302647 (46 \#1791)}

\bibitem[Lun73]{Luna}
Domingo Luna, \emph{Slices \'etales}, Sur les groupes alg\'ebriques, Soc. Math.
  France, Paris, 1973, pp.~81--105. Bull. Soc. Math. France, Paris, M\'emoire
  33. \MR{0342523 (49 \#7269)}

\bibitem[Mat89]{Matsumura-ringtheory}
Hideyuki Matsumura, \emph{Commutative ring theory}, second ed., Cambridge
  Studies in Advanced Mathematics, vol.~8, Cambridge University Press,
  Cambridge, 1989, Translated from the Japanese by M. Reid. \MR{1011461
  (90i:13001)}

\bibitem[MFK94]{GIT}
D.~Mumford, J.~Fogarty, and F.~Kirwan, \emph{Geometric invariant theory}, third
  ed., Ergebnisse der Mathematik und ihrer Grenzgebiete (2) [Results in
  Mathematics and Related Areas (2)], vol.~34, Springer-Verlag, Berlin, 1994.
  \MR{1304906 (95m:14012)}

\bibitem[SGA70a]{SGA3-1}
\emph{Sch\'emas en groupes. {I}: {P}ropri\'et\'es g\'en\'erales des sch\'emas
  en groupes}, S\'eminaire de G\'eom\'etrie Alg\'ebrique du Bois Marie 1962/64
  (SGA 3). Dirig\'e par M. Demazure et A. Grothendieck. Lecture Notes in
  Mathematics, Vol. 151, Springer-Verlag, Berlin-New York, 1970. \MR{0274458
  (43 \#223a)}

\bibitem[SGA70b]{SGA3-2}
\emph{Sch\'emas en groupes. {II}: {G}roupes de type multiplicatif, et structure
  des sch\'emas en groupes g\'en\'eraux}, S\'eminaire de G\'eom\'etrie
  Alg\'ebrique du Bois Marie 1962/64 (SGA 3). Dirig\'e par M. Demazure et A.
  Grothendieck. Lecture Notes in Mathematics, Vol. 152, Springer-Verlag,
  Berlin-New York, 1970. \MR{0274459 (43 \#223b)}

\bibitem[{Sta}]{stacks}
The {Stacks Project Authors}, \emph{{\itshape Stacks Project}},
  \url{http://stacks.math.columbia.edu}.

\bibitem[Tem04]{Temkin-local-properties}
Michael Temkin, \emph{On local properties of non-{A}rchimedean analytic spaces.
  {II}}, Israel J. Math. \textbf{140} (2004), 1--27. \MR{2054837 (2005c:14030)}

\bibitem[Tha96]{Thaddeus}
Michael Thaddeus, \emph{Geometric invariant theory and flips}, J. Amer. Math.
  Soc. \textbf{9} (1996), no.~3, 691--723. \MR{1333296 (96m:14017)}

\bibitem[W{\l}o00]{W-Cobordism}
Jaros{\l}aw W{\l}odarczyk, \emph{Birational cobordisms and factorization of
  birational maps}, J. Algebraic Geom. \textbf{9} (2000), no.~3, 425--449.
  \MR{1752010 (2002d:14019)}

\bibitem[W{\l}o03]{W-Factor}
\bysame, \emph{Toroidal varieties and the weak factorization theorem}, Invent.
  Math. \textbf{154} (2003), no.~2, 223--331. \MR{2013783 (2004m:14113)}

\end{thebibliography}
\printindex
\end{document}